\documentclass[12pt]{amsart}
\usepackage{amssymb}
\usepackage{longtable}
\usepackage{hyperref}
\usepackage[total={6.25truein,9truein},centering]{geometry}
\usepackage{tikz}

\allowdisplaybreaks

\numberwithin{equation}{section}

\newtheorem{theorem}[equation]{Theorem}
\newtheorem{lemma}[equation]{Lemma}
\newtheorem{proposition}[equation]{Proposition}
\newtheorem{corollary}[equation]{Corollary}

\newtheorem{theoremintro}{Theorem}

\theoremstyle{definition}

\theoremstyle{remark}
\newtheorem{example}[equation]{Example}
\newtheorem{remark}[equation]{Remark}
\newtheorem{definition}[equation]{Definition}
\newtheorem*{acknowledgments}{Acknowledgments}

\newcommand{\FF}{\mathbb{F}}
\newcommand{\ZZ}{\mathbb{Z}}

\newcommand{\NN}{\mathbb{N}}
\newcommand{\TT}{\mathbb{T}}

\newcommand{\KK}{\mathbb{K}}

\newcommand{\bA}{\mathbf{A}}

\newcommand{\bm}{\mathbf{m}}

\newcommand{\bS}{\mathbf{S}}

\newcommand{\cB}{\mathcal{B}}

\newcommand{\cL}{\mathcal{L}}

\newcommand{\bsi}{\boldsymbol{i}}

\DeclareMathOperator{\GL}{GL}

\DeclareMathOperator{\Mat}{Mat}

\DeclareMathOperator{\Res}{Res}

\newcommand{\sep}{\mathrm{sep}}

\newcommand{\tB}{\widetilde{B}}

\newcommand{\tE}{\widetilde{E}}
\newcommand{\tpi}{\widetilde{\pi}}

\newcommand{\tx}{\tilde{x}}
\newcommand{\ty}{\tilde{y}}
\newcommand{\tX}{\widetilde{X}}

\newcommand{\power}[2]{{#1 [\![ #2 ]\!]}}
\newcommand{\laurent}[2]{{#1 (\!( #2 )\!)}}

\newcommand{\norm}[1]{\lvert #1 \rvert}
\newcommand{\dnorm}[1]{\lVert #1 \rVert}

\newcommand{\assign}{\mathrel{\vcenter{\baselineskip0.5ex \lineskiplimit0pt
                     \hbox{\scriptsize.}\hbox{\scriptsize.}}}%
                     =}

\begin{document}

\title[Effective rigid analytic trivializations for Drinfeld modules]{Effective rigid analytic trivializations \\ for Drinfeld modules}

\author{Chalinee Khaochim}
\address{Phetchaburi Rajabhat University Demonstration School, Phetchaburi Rajabhat University, Phetchaburi 76000, Thailand}
\email{chalinee.kha@mail.pbru.ac.th}

\author{Matthew A. Papanikolas}
\address{Department of Mathematics, Texas A{\&}M University, College Station,
TX 77843, U.S.A.}
\email{papanikolas@tamu.edu}


\date{April 15, 2022}

\begin{abstract}
We develop tools for constructing rigid analytic trivializations for Drinfeld modules as infinite products of Frobenius twists of matrices, from which we recover the rigid analytic trivialization given by Pellarin in terms of Anderson generating functions. One advantage is that these infinite products can be obtained from only a finite amount of initial calculation, and consequently we obtain new formulas for periods and quasi-periods, similar to the product expansion of the Carlitz period. We further link to results of Gekeler and Maurischat on the $\infty$-adic field generated by the period lattice.
\end{abstract}


\maketitle

\section{Introduction} \label{S:Intro}
Rigid analytic trivializations were originally defined by Anderson~\cite[Thm.~4]{Anderson86} as tools for determining whether a $t$-module is uniformizable, i.e., that its exponential function is surjective. It was subsequently discovered by Anderson and Pellarin that specializations of rigid analytic trivializations could be used to recover the periods and quasi-periods of a Drinfeld module (see~\cite[\S 2.6]{Goss94}, \cite[\S 4]{Pellarin08}). This specialization phenomenon has led to advances in the transcendence theory of periods and quasi-periods, e.g., see \cite{ABP04}, \cite{CP11}, \cite{CP12}, \cite{Pellarin08}. Given their centrality to the arithmetic of function fields, the focus of the present paper is to investigate new and direct ways of constructing rigid analytic trivializations for Drinfeld modules.

Our motivating example is that of the \emph{Carlitz module}. Let $\FF_q$ be a finite field with~$q$ elements, let $A \assign \FF_q[\theta]$ be the polynomial ring in $\theta$ over $\FF_q$, and let $k\assign \FF_q(\theta)$ be its fraction field. We let $\bA \assign \FF_q[t]$ be the polynomial ring in a variable~$t$ independent from~$\theta$. For any $k$-algebra $R$, the Carlitz module $C$ over $R$ is the $\bA$-module structure on $R$ determined by setting $C_t(x) \assign \theta x + x^q$ for $x \in R$.  If we let $k_\infty \assign \laurent{\FF_q}{1/\theta}$ be the completion of $k$ at its infinite place, and take $\KK$ for the completion of an algebraic closure of $k_\infty$, then the \emph{Carlitz exponential} $\exp_C : \KK \to \KK$ is entire, $\FF_q$-linear, and surjective, and it uniformizes the Carlitz module over~$\KK$.  The kernel of $\exp_C$ is the discrete $A$-submodule of $\KK$ of rank~$1$ generated by the \emph{Carlitz period},
\begin{equation} \label{E:tpi}
    \tpi = -(-\theta)^{q/(q-1)} \prod_{n=1}^{\infty} \Bigl( 1- \theta^{1-q^n} \Bigr)^{-1} \in k_{\infty} \bigl( (-\theta)^{1/(q-1)} \bigr),
\end{equation}
where we have fixed a $(q-1)$-st root of $-\theta$. Throughout the arithmetic of function fields, $\tpi$ plays the role of $2\pi i$ in characteristic~$0$. This product formula for $\tpi$ was essentially first derived by Carlitz~\cite[Thm.~5.1]{Carlitz35}, and for more information on the Carlitz module and its exponential function, see \cite[Ch.~3]{Goss}, \cite[Ch.~2]{Thakur}.

The rigid analytic trivialization of $C$ is the \emph{Anderson-Thakur function}
\begin{equation} \label{E:omega}
    \omega_C(t) \assign (-\theta)^{1/(q-1)} \prod_{n=0}^{\infty} \biggl( 1 - \frac{t}{\theta^{q^n}} \biggr)^{-1} \in \TT,
\end{equation}
where $\TT$ is the Tate algebra in $\power{\KK}{t}$ consisting of power series that converge on the closed unit disk in~$\KK$. In~\cite[\S 2.5]{AndThak90}, Anderson and Thakur showed that $\omega_C$ satisfies a number of important properties. For $n \in \ZZ$, if we define the Frobenius twist $f \mapsto f^{(n)} : \laurent{\KK}{t} \to \laurent{\KK}{t}$ by $\sum c_i t^i \mapsto \sum c_i^{q^n} t^i$, then
\begin{equation} \label{E:omegafneq}
  \omega_C^{(1)} = (t-\theta)\omega_C.
\end{equation}
Moreover, $\omega_C$ extends to a meromorphic function on all of~$\KK$, and it has a simple pole at $t=\theta$, where we readily verify from~\eqref{E:tpi} that
\begin{equation} \label{E:omegares}
 \Res_{t=\theta} \omega_C = -\tpi.
\end{equation}
The Anderson-Thakur function further plays a central role in special values of Pellarin $L$-series (e.g., see \cite{Pellarin12}, \cite{Perkins14}).

Letting $\KK[\tau]$ be the ring of twisted polynomials in the $q$-th power Frobenius $\tau$, a \emph{Drinfeld module} of rank $r$ is an $\FF_q$-algebra homomorphism $\phi : \bA \to \KK[\tau]$ determined by
\begin{equation} \label{E:phiintro}
 \phi_t = \theta + A_1 \tau + \dots + A_r \tau^r, \quad A_r \neq 0.
\end{equation}
Like the Carlitz module, $\phi$ is uniformized by an exponential function $\exp_{\phi} : \KK \to \KK$, whose kernel $\Lambda_{\phi}$ is a discrete free $A$-submodule of $\KK$ of rank~$r$. We call $\Lambda_{\phi}$ the \emph{period lattice} of $\phi$. For $\pi \in \Lambda_{\phi}$, we define the \emph{Anderson generating function}
\[
    f_{\phi}(\pi;t) \assign \sum_{m=0}^\infty \exp_{\phi} \biggl( \frac{\pi}{\theta^{m+1}} \biggr) t^m \in \TT.
\]
Initially defined by Anderson~\cite[\S 3.2]{Anderson86}, these functions satisfy a number of useful properties which we summarize in \S\ref{S:Prelim}.  Most notably,
\[
  \theta f_{\phi}(\pi;t) + A_1 f_{\phi}(\pi;t)^{(1)} + \cdots + A_r f_{\phi}(\pi;t)^{(r)} = t f_{\phi}(\pi;t),
\]
and $f_{\phi}(\pi;t)$ is a meromorphic function on $\KK$ with a simple pole at $t=\theta$ (if $\pi \neq 0$) and
\begin{equation} \label{E:fpires}
  \Res_{t=\theta} f_{\phi}(\pi;t) = -\pi.
\end{equation}
There is an obvious parallel with \eqref{E:omegafneq} and~\eqref{E:omegares}, and in fact $\omega_C$ is the Anderson generating function for~$\tpi$ on the Carlitz module (see \cite[Prop.~5.1.3]{ABP04}).

For an $A$-basis $\pi_1, \dots, \pi_r$ of $\Lambda_{\phi}$, we set $f_j \assign f_{\phi}(\pi_j;t)$ for each $j$, and we let
\begin{equation} \label{E:Upsilon}
\Upsilon  \assign \begin{pmatrix}
f_1    &   f_2         &  \cdots      &   f_r \\
f_1^{(1)} &   f_2^{(1)}   &  \cdots      &   f_r^{(1)} \\
\vdots    &   \vdots      &  \ddots      &   \vdots  \\
f_1^{(r-1)} &   f_2^{(r-1)}   &  \cdots      &   f_r^{(r-1)}
\end{pmatrix}.
\end{equation}
Pellarin~\cite[\S 4.2]{Pellarin08} showed that $\Upsilon$ is invertible in $\GL_r(\TT)$, and moreover, he observed that (i) by~\eqref{E:fpires} the negatives of the residues at $t=\theta$ of the entries of the first row yield the periods $\pi_1, \dots, \pi_r$, and (ii) based on calculations of Gekeler~\cite[Rem.~2.7]{Gekeler89}, for $1 \leqslant i \leqslant r-1$, the value $f_j^{(i)}(\theta)$ is a strictly reduced quasi-period for $\phi$ associated to $\pi_j$.  Furthermore, he proved that
\begin{equation} \label{E:UpsRAT}
  \Upsilon^{(1)} = \Theta \Upsilon,
\end{equation}
where we apply the Frobenius twist entry-wise and where
\begin{equation} \label{E:Theta}
\Theta \assign \begin{pmatrix}
0    &   1   &  \cdots  &   0 \\
\vdots &  \vdots &  \ddots &  \vdots  \\
0    &   0   &  \cdots  &   1 \\
(t-\theta)/A_r & -A_1/A_r  &  \cdots  & -A_{r-1}/A_r\\
\end{pmatrix}.
\end{equation}
The matrix $\Theta$ arises naturally from the $t$-motive associated to $\phi$, and by definition, \eqref{E:UpsRAT} makes $\Upsilon$ into a \emph{rigid analytic trivialization} for $\phi$ (see \S\ref{S:Prelim} for details).

\begin{remark}
(1) This same specialization property that uses rigid analytic trivializations to supply periods and quasi-periods holds for other $t$-modules as well.  For rank~$1$ objects, one can consult~\cite[\S 6]{ABP04}, \cite[\S 4]{BP02}, \cite{Sinha97}, for instances involving the geometric $\Gamma$-function, and~\cite[Thm.~4.6]{GreenP18} for certain rank~$1$ Drinfeld modules over more general rings~$\bA$. For general $t$-modules, see~\cite[\S 3]{GazdaMaurischat22}, \cite[\S 6]{Green22}, \cite[\S 5]{Maurischat22}, for periods, and \cite[\S 4]{NP21} for periods and quasi-periods. In all of these cases, rigid analytic trivializations are obtained through possibly higher dimensional versions of Anderson generating functions.

(2) In the present paper we study only rigid analytic trivializations associated to the $t$-motive of a Drinfeld module, as opposed to the rigid analytic trivializations associated to its dual $t$-motive.  However, as seen in~\cite[\S 3.4]{CP12}, \cite[Ex.~2.5.16]{HartlJuschka20}, the two theories are related, and moreover, using a theorem of Hartl and Juschka~\cite[Thm.~2.5.13]{HartlJuschka20}, we can transform one to the other for general abelian and $\bA$-finite $t$-modules~\cite[Thm.~4.4.14]{NP21}.
\end{remark}

There is one noticeable advantage of $\omega_C(t)$ over the more general Anderson generating functions $f_{\phi}(\pi;t)$: the definition of $f_{\phi}(\pi;t)$ ostensibly presupposes knowledge of $\pi$ itself, whereas the product expansion in~\eqref{E:omega} is independent of knowing $\tpi$ in advance. The coefficients of $f_{\phi}(\pi;t)$ form a $t$-division sequence of $t$-power torsion on~$\phi$, but obtaining $f_{\phi}(\pi;t)$ exactly from these coefficients would require the precise selection of this infinite sequence. Thus unfortunately, in order to use $f_{\phi}(\pi;t)$ to retrieve~$\pi$ or the quasi-periods associated to~$\pi$, we have run into a chicken-and-egg problem.

One goal of the present paper is to construct a rigid analytic trivialization $\Upsilon$ for $\phi$ from a full system of Anderson generating functions $f_1, \dots, f_r$, associated to an $A$-basis $\pi_1, \dots, \pi_r$ of $\Lambda_{\phi}$, \emph{without} using $\pi_1, \dots, \pi_r$ as initial inputs. Much like for $\omega_C$, we will achieve this through an infinite product, but of matrices, which can be constructed in an effective manner, i.e., by utilizing only a \emph{finite} amount of initial computation (see Theorems~\ref{TI:PiRAT} and~\ref{TI:Pi}). In this way we can obtain $\Upsilon$ efficiently, together with formulas for the periods and quasi-periods of our Drinfeld module, from first principles.

The outline of this construction is as follows. After establishing definitions and prior results in \S\ref{S:Prelim}, we construct a matrix $B \in \Mat_r(\KK[t]) \cap \GL_r(\TT)$ in \S\ref{S:MatEst} so that 
\begin{equation} \label{E:BThetaineq}
  \dnorm{B^{-1} \Theta^{-1} B^{(1)} - I} < 1,
\end{equation}
where $\dnorm{\,\cdot\,}$ is the extension of the Gauss norm on $\TT$ to $\Mat_r(\TT)$ and $I$ is the $r\times r$ identity matrix (see Theorem~\ref{T:B}). The entries of $B$ are obtained by systematic selection of a basis of $t^N$-torsion points $\xi_1, \dots, \xi_r \in \phi[t^N]$, where $N \geqslant 1$ is determined by the degrees of the coefficients $A_1, \dots, A_r$ in~\eqref{E:phiintro} using estimates from~\cite{EP14} (see Remark~\ref{R:N}).  We choose these $t^N$-torsion points recursively through analysis of the Newton polygon of $\phi_t(x) \in \KK[x]$ (see Propositions~\ref{P:algo} and~\ref{P:good-xi}), and moreover, $\phi_{t^{N-1}}(\xi_1), \dots, \phi_{t^{N-1}}(\xi_r) \in \phi[t]$ form a \emph{strict basis} of the $t$-torsion module $\phi[t]$, in that the degrees of these elements match the slopes of the Newton polygon of $\phi_t(x)$ in a prescribed way (see Definition~\ref{D:Strict}). By letting
\[
  h_j \assign \phi_{t^{N-1}} (\xi_j) + \phi_{t^{N-2}}(\xi_j)t + \cdots + \xi_j t^{N-1},
\]
the matrix
\begin{equation} \label{E:Bintro}
  B \assign \bigl( h_j^{(i-1)} \bigr) \in \Mat_r(\KK[t])
\end{equation}
is an element of $\GL_r(\TT)$ (see Proposition~\ref{P:detB}). It further satisfies~\eqref{E:BThetaineq} by Theorem~\ref{T:B}, which leads to the following result (stated later as Corollary~\ref{C:PiRAT}), producing 
a rigid analytic trivialization for~$\phi$ via an infinite product of twists of matrices.

\begin{theoremintro} \label{TI:PiRAT}
Continuing with notation as above, the infinite product
\[
  \Pi \assign B \prod_{n=0}^\infty \Bigl( B^{-1} \Theta^{-1} B^{(1)} \Bigr)^{(n)}
\]
converges with respect to the Gauss norm on $\Mat_r(\TT)$ and lies in~$\GL_r(\TT)$.  Moreover,
\[
  \Pi^{(1)} = \Theta \Pi,
\]
and so $\Pi$ is a rigid analytic trivialization for~$\phi$.
\end{theoremintro}

Because such a product of matrices depends on the order of the factors, we note that in Theorem~\ref{TI:PiRAT} and elsewhere each successive term of the infinite product is multiplied on the right. That is, the product is $(B^{-1} \Theta^{-1} B^{(1)})( B^{-1} \Theta^{-1} B^{(1)})^{(1)}( B^{-1} \Theta^{-1} B^{(1)})^{(2)} \cdots$.

In the case of the Carlitz module, one checks that $\Theta = t-\theta$ and $B = (-\theta)^{1/(q-1)}$ in Theorem~\ref{TI:PiRAT}, from which we see that $\Pi = \omega_C$. To what extent can we use $\Pi$ to recover $\Upsilon$ for general Drinfeld modules $\phi$? Like the Carlitz module, Drinfeld modules of rank~$1$ over more general rings possess product expansions for their rigid analytic trivializations \cite[\S 3]{ANT17}, \cite[\S 4]{GreenP18}. But in general we can in fact use $\xi_1, \dots, \xi_r$ to construct an $A$-basis $\pi_1, \dots, \pi_r$ of $\Lambda_\phi$, and then apply \cite[Thm.~6.13]{EP14} to show that $\Pi$ is the same as $\Upsilon$ in~\eqref{E:Upsilon}. The following is our main result in these directions (stated later with additional details in Theorem~\ref{T:Pi}).

\begin{theoremintro} \label{TI:Pi}
Choose $N \geqslant 1$ and $\xi_1, \dots, \xi_r \in \phi[t^N]$ as in Proposition~\ref{P:good-xi}. Let $B = ( h_j^{(i-1)}) \in \GL_r(\TT)$ be defined as in \eqref{E:Bintro}, and construct the rigid analytic trivialization $\Pi$ for $\phi$ as in Theorem~\ref{TI:PiRAT}. Letting $\pi_j \assign \theta^N \log_{\phi}(\xi_j)$, the quantities $\pi_1, \dots, \pi_r$ form an $A$-basis of $\Lambda_{\phi}$, and moreover
\[
  \Pi = \Upsilon,
\]
where $\Upsilon$ is defined with respect to $\pi_1, \dots, \pi_r$ in~\eqref{E:Upsilon}.
\end{theoremintro}

In addition to providing identities for periods and quasi-periods, Theorems~\ref{TI:PiRAT} and~\ref{TI:Pi} and their proofs lead to precise descriptions of the field generated by the period lattice over the field of definition of the Drinfeld module in Corollary~\ref{C:fieldexts}. This recovers a result of Maurischat~\cite[Thm.~3.1]{Maurischat19a} in rank~$2$, and parts of results of Gekeler in arbitrary rank~\cite[\S 2]{Gekeler19}. It also wraps up a picture started in~\cite[Thm.~5.3]{EP13}.

\begin{remark}
As pointed out by one referee, the results in Theorems~\ref{TI:PiRAT} and~\ref{TI:Pi} could in principle be extended to Anderson $t$-modules which are abelian and rigid analytically trivial. Further investigation would be necessary to make this precise, though perhaps descriptions of rigid analytic trivializations and Anderson generating functions from~\cite{GazdaMaurischat22}, \cite{Green22}, \cite{HartlJuschka20}, \cite{Maurischat22}, \cite{NP21}, would be helpful. One thing to note is that the proofs in the present paper, and especially the proof of the convergence of the product in Theorem~\ref{TI:PiRAT}, rely heavily on the explicit nature of $\Theta$ given in~\eqref{E:Theta}.
\end{remark}

In \S\ref{S:Rank2} we investigate the case of rank~$2$ in more detail. Our findings dovetail with Maurischat's theorem~\cite[Thm.~3.1]{Maurischat19a}, which we summarize in Theorem~\ref{T:Rank2}. Then in Examples~\ref{Ex:Car} and~\ref{Ex:A} we approximate the matrices $B$ and $\Pi$ for specific rank~$2$ Drinfeld modules and use these approximations to calculate periods and quasi-periods.

\begin{acknowledgments}
The authors thank Q.~Gazda and A.~Maurischat for valuable comments on a previous version of this paper. They further thank the referees for a number of suggestions that improved some mathematical statements and greatly simplified exposition.
\end{acknowledgments}

\section{Preliminaries} \label{S:Prelim}
The following notation will be used throughout the paper:

\begin{longtable}{p{2.25truecm}@{\hspace{5pt}$=$\hspace{5pt}}p{11truecm}}
$\FF_q$ & finite field with $q$ elements, $q$ a power of a fixed prime $p$. \\
$A$ & $\FF_q[\theta]$, the polynomial ring in $\theta$ over $\FF_q$. \\
$k$ & $\FF_q(\theta)$, the fraction field of $A$. \\
$k_\infty$ & $\laurent{\FF_q}{1/\theta}$, the completion of $k$ at its infinite place. \\
$\KK$ & the completion of an algebraic closure of $k_\infty$. \\
$\deg$ & $-v_{\infty}$, where $v_{\infty}$ is the $\infty$-adic valuation on $\KK$, $\deg \theta =1$. \\
$\bA$ & $\FF_q[t]$, the polynomial ring in $t$ over $\FF_q$, $t$ independent from $\theta$. \\
$\TT$ & the Tate algebra of $\power{\KK}{t}$ on the closed unit disk.\\
$\TT_\theta$ & the Tate algebra of $\power{\KK}{t}$ on the closed disk of radius $\norm{\theta}$.\\
$F^{\sep}$ & a separable algebraic closure of a field $F$.\\
$\Mat_{m \times n}(R)$ & for a ring $R$, the left $R$-module of $m \times n$ matrices.\\
$\Mat_d(R)$ & $\Mat_{d\times d}(R)$.\\
$[B]_{ij}$ & the $(i,j)$-entry of a matrix $B$.
\end{longtable}

The absolute value $\norm{\,\cdot\,}$ on $k_\infty$ is chosen so that $|\theta|=q$, its valuation satisfies $v_{\infty}(\theta)=-1$, and we let $\deg \assign -v_{\infty}$. Then $\norm{\,\cdot\,}$, $v_\infty$, and $\deg$ extend uniquely to $\KK$.

\textit{Drinfeld modules}. For fundamental properties of Drinfeld modules, see Goss~\cite[Ch.~3--4]{Goss} or Thakur \cite[Ch.~2]{Thakur}.  For the $q$-th power Frobenuis map $\tau: \KK \to \KK$ ($z\mapsto z^q$), the ring of twisted polynomials $\KK[\tau]$ in $\tau$ over $\KK$ satisfies $\tau c = c^q \tau$ for all $c \in \KK$. A \emph{Drinfeld module of rank $r$} over $\KK$ is an $\FF_q$-algebra homomorphism $\phi:\bA \to \KK[\tau]$ determined by
\begin{equation}  \label{E:DMdef}
 \phi_t = \theta + A_1 \tau + \dots +  A_r \tau^r,\quad A_r \neq 0.
\end{equation}
As usual, we obtain an $\bA$-module structure on $\KK$ induced by $\phi$ by the action $a \cdot x \assign \phi_a(x)$ for $a\in\bA$, $x\in\KK$. For any $a\in\bA$, we take $\phi[a] \assign \{x\in\KK : \phi_a(x) = 0 \}$ to be the $\bA$-submodule of $a$-torsion points on~$\phi$.  Then $\phi[a] \cong (\bA/(a))^r$ as $\bA$-modules.

The \emph{exponential} of $\phi$ is defined to be the entire, surjective, $\FF_q$-linear power series,
\begin{equation} \label{E:expdef}
\exp_\phi(z) = \sum_{n=0}^\infty \alpha_n z^{q^n}, \quad \alpha_0 =1,\ \alpha_n\in\KK,
\end{equation}
satisfying $\exp_\phi(a(\theta)z) = \phi_a(\exp_\phi(z))$ for every $a\in\bA$. Letting $\Lambda_\phi \assign \ker \exp_\phi$, one finds that $\Lambda_\phi$ is a discrete, free $A$-module of rank $r$ inside $\KK$. We call $\Lambda_\phi$ the \emph{period lattice} of $\phi$, and any element of $\Lambda_\phi$ a \emph{period} of $\phi$. The \emph{logarithm} $\log_{\phi}(z)$ of $\phi$ is the formal inverse of $\exp_\phi(z)$ with respect to composition, and it has a finite radius of convergence $P_\phi$ on~$\KK$. By \cite[Prop.~4.14.2]{Goss},
\begin{equation} \label{E:lograd}
 P_{\phi} = \min\{ |\pi| : \pi \in \Lambda_{\phi},\, \pi \neq 0 \}.
\end{equation}

\textit{Tate algebras}. The Tate algebra of power series converging on the closed unit disk of~$\KK$,
\[
\TT \assign \biggl\{ \sum c_i t^i  \in \power{\KK}{t} : |c_i| \to 0,\, \textup{for $i \to \infty$} \biggr\},
\]
is a complete normed $\KK$-algebra with respect to the Gauss norm $\dnorm{\,\cdot\,}$, which is defined by $\dnorm{\sum c_i t^i} \assign \sup_i |c_i| = \max_i |c_i|$. We recall that $u=\sum c_i t^i \in \TT$ is in $\TT^{\times}$ if and only if (i) $c_0 \neq 0$ and (ii) we can write $u = c_0(1+g)$ for $g \in \TT$ with $\dnorm{g} < 1$ (see \cite[Lem.~2.9.1]{Anderson86}, \cite[Cor.~2.2.4]{FresnelvdPut}). For $\eta \in \KK^{\times}$, we further define the Tate algebra
\[
\TT_\eta \assign \biggl\{ \sum c_i t^i \in \power{\KK}{t} : |\eta|^i \cdot |c_i| \to 0,\, \textup{for $i \to \infty$} \biggr\},
\]
consisting of functions that converge on the closed disk of radius~$|\eta|$.  For more details on the theory of Tate algebras, see~\cite[\S 2]{FresnelvdPut}. For a matrix $F = (f_{ij}) \in \Mat_{r\times s}(\TT)$, we set $\dnorm{F} \assign \max_{i,j} \dnorm{f_{ij}}$, making $\Mat_{r\times s}(\TT)$ into a complete normed $\TT$-module.

For any $f=\sum c_i t^i \in \laurent{\KK}{t}$ and $n\in\ZZ$, the $n$-th \emph{Frobenius twist} of $f$ is $f^{(n)} \assign \sum c_i^{q^n} t^i$. For $F = (f_{ij}) \in \Mat_{r\times s}(\laurent{\KK}{t})$, we take $F^{(n)} \assign (f_{ij}^{(n)})$. For $\delta > 0$ and any $f \in \TT_{\theta^{\delta/q}}$, we have $f^{(1)} \in \TT_{\theta^{\delta}}$.  In particular,
\begin{equation} \label{E:twistTate}
  f \mapsto f^{(1)} : \TT_{\theta^{1/q}} \to \TT_{\theta}.
\end{equation}
For $\Delta = b_0 + b_1\tau + \cdots + b_{\ell} \tau^{\ell} \in \KK[\tau]$ and $f \in \TT$, we define
\begin{equation} \label{E:Deltaop}
  \Delta(f) \assign b_0 f + b_1 f^{(1)} + \cdots + b_{\ell} f^{(\ell)},
\end{equation}
thus making $\Delta$ into an $\bA$-linear endomorphism of $\TT$.

\textit{Anderson generating functions}. We continue with our Drinfeld module $\phi$ in~\eqref{E:DMdef}. For $u\in\KK$, the Anderson generating function for $\phi$ associated to $u$ is defined by
\begin{equation} \label{E:AGFdef}
f_\phi(u;t) \assign \sum_{m=0}^\infty \exp_\phi \biggl(\frac{u}{\theta^{m+1}}\biggr)t^m \in \TT.
\end{equation}
Pellarin~\cite[\S 4.2]{Pellarin08} exhibited a partial fraction decomposition,
\[
f_\phi(u;t) = \sum_{n=0}^\infty \frac{\alpha_n u^{q^n}}{\theta^{q^n}-t},
\]
where $\alpha_n$ are the coefficients of $\exp_\phi$ from~\eqref{E:expdef}. From this decomposition we see that $f_\phi(u;t)$ extends to a meromorphic function on $\KK$ with simple poles (when $u\neq 0$) at $t = \theta^{q^n}$, $n = 0$, $1, \ldots$, with respective residues $\Res_{t=\theta^{q^n}} f_\phi(u;t) = -\alpha_n u^{q^n}$. In particular,
\[
\Res_{t=\theta} f_\phi(u;t) = -u.
\]
It follows from~\eqref{E:AGFdef} and the functional equation for $\exp_{\phi}(z)$ that
\begin{equation} \label{E:AGFscalar}
    \phi_t(f_{\phi}(u;t)) = \theta f_{\phi}(u;t) + A_1 f_{\phi}(u;t)^{(1)} + \cdots + A_r f_{\phi}(u;t)^{(r)} = f_{\phi}(\theta u;t),
\end{equation}
and so for each $a \in \bA$, $\phi_a(f_{\phi}(u;t)) = f_{\phi}(a(\theta) u;t)$. Another fundamental property is that
\begin{equation} \label{E:AGFfnq}
  \phi_t(f_{\phi}(u;t)) = t f_{\phi}(u;t) + \exp_{\phi}(u).
\end{equation}
It follows that if $\pi \in \Lambda_{\phi}$, then for all $a \in \bA$, we have $\phi_a(f_{\phi}(\pi;t)) = a f_{\phi}(\pi;t)$.  The partial fraction decomposition of $f_{\phi}(\pi;t)$ implies that it is an element of $\TT_{\eta}$ for any $\eta \in \KK$ with $|\eta| < |\theta|$.  Thus by~\eqref{E:twistTate} we find that $f_{\phi}(\pi;t)^{(1)} \in \TT_{\theta}$, and in particular $f_{\phi}(\pi;t)^{(i)}$ is well-defined at $t=\theta$ for all $i \geqslant 1$.

\textit{Logarithm deformations}. We fix a Drinfeld module $\phi$ of rank~$r$ as in~\eqref{E:DMdef}. For $\xi \in \KK$, El-Guindy and the second author defined a series $\cL_{\phi}(\xi;t)$ in~\cite{EP14}, which is a deformation of $\log_\phi(\xi)$ and is related to Anderson generating functions.  This series is defined using \emph{shadowed partitions} defined in~\cite{EP13} as follows. For $n$, $r\in\NN$, we let $P_r(n)$ be the set of $r$-tuples $(S_1,S_2,\ldots,S_r)$ such that (i) for each $i$, $S_i \subseteq \{0,1,\ldots,n-1\}$, and (ii) the sets $\{S_i + j :1\leqslant i\leqslant r, 0 \leqslant j\leqslant i-1\}$ form a partition of $\{0,1,\ldots,n-1\}$. For $n\in\NN$, define
\begin{equation} \label{E:Bndef}
\cB_n(t) \assign \sum_{\bS \in P_r(n)} \prod_{i=1}^r \prod_{j\in S_i} \frac{A_i^{q^j}}{t-\theta^{q^{i+j}}} \in \KK(t).
\end{equation}
Let $N(\phi) \assign \{1\leqslant i\leqslant r : A_i \neq0\}$. For each $n \in N(\phi)$, let
\begin{equation} \label{E:mu_n}
    \mu_n \assign \frac{\deg A_n -q^n}{q^n -1},
\end{equation}
and let
\begin{equation} \label{E:radius}
   R_\phi \assign |\theta|^{-\mu_m},
\end{equation}
where $m$ is the smallest index in $N(\phi) \assign \{1\leqslant i\leqslant r : A_i \neq 0\}$ such that $\mu_m \geqslant \mu_i$ for every $i\in N(\phi)$.

\begin{remark} \label{R:Rphi}
It was shown in \cite[Rem.~6.11, Thm.~6.13(b)]{EP14} that $R_{\phi} \leqslant P_{\phi}$, where $P_{\phi}$ is the radius of convergence of $\log_\phi(z)$ from~\eqref{E:lograd} and that $R_{\phi}=P_{\phi}$ if $\mu_m > \mu_i$ for all $i \neq m$. We will see in Corollary~\ref{C:Rphi} that in fact $R_{\phi} = P_{\phi}$ in all cases.
\end{remark}

Assuming $|\xi| < R_{\phi}$, we set
\begin{equation} \label{E:Lxit}
\cL_\phi(\xi;t) \assign \sum_{n=0}^\infty \cB_n(t) \xi^{q^n} \in \TT,
\end{equation}
which converges in~$\TT$ with respect to the Gauss norm and, as a function of $t$, converges on the open disk of radius $|\theta|^q$ in $\KK$~\cite[Prop.~6.10]{EP14}.  Furthermore, if $\exp_{\phi}(u) = \xi$ and $|u| < R_{\phi}$, then by~\cite[Thm.~6.13]{EP14},
\[
  \cL_{\phi}(\xi; \theta) = \log_{\phi}(\xi) = u,
\]
and moreover, we recover the Anderson generating function for~$u$, as
\begin{equation} \label{E:LxitAGF}
  \cL_{\phi}(\xi;t) = -(t-\theta) f_{\phi}(u;t).
\end{equation}

\textit{t-motives for Drinfeld modules}. Anderson originally defined $t$-motives in \cite{Anderson86}, which we briefly review. The ring $\KK[t, \tau]$ is the polynomial ring in $t$ and $\tau$ with coefficients in $\KK$ subject to the following relations,
\[
tc = ct, \quad t \tau = \tau t, \quad  \tau c = c^q \tau, \quad c \in \KK.
\]
A \emph{$t$-motive} $M$ is a left $\KK[t, \tau]$-module that is free and finitely generated as a left $\KK[\tau]$-module and for which there is $\ell \in \NN$ with $(t - \theta)^\ell (M/\tau M) = \{0\}$. The rank~$d$ of $M$ as a left $\KK[\tau]$-module is the \emph{dimension} of $M$.  Diverging from Anderson's usage somewhat, we say that $M$ is \emph{abelian} if $M$ is also free and finitely generated as a left $\KK[t]$-module.  In this case the \emph{rank} of $M$ is its rank as a $\KK[t]$-module.

Given our Drinfeld module $\phi : \bA \to \KK[\tau]$, as in~\eqref{E:DMdef}, the \emph{$t$-motive associated to~$\phi$}, denoted $M(\phi)$, is defined as follows: let $M(\phi) \assign \KK[\tau]$ and make $M(\phi)$ into a left $\KK[t]$-module by setting $c t^i \cdot m \assign c m  \phi_{t^{i}}$, for $m \in M(\phi)$, $c \in \KK$.
The $t$-motive $M(\phi)$ is abelian of rank~$r$ and dimension~$1$.

\textit{Rigid analytic trivializations}. Anderson~\cite[\S 2.3]{Anderson86} defined the notion of an abelian $t$-motive to be rigid analytically trivial, which is equivalent to the following definition. Let $\bm \in \Mat_{r\times 1}(M)$ comprise a $\KK[t]$-basis for $M$, and let $\Theta \in  \Mat_r (\KK[t])$ represent multiplication by $\tau$ on $M$ with respect to $\bm$, i.e., $\tau \bm = \Theta \bm$.  It is known that $\det \Theta = c(t-\theta)^d$ for some $c \in \KK^{\times}$, where $d$ is the dimension of $M$ (see~\cite[\S 3.2]{Anderson86} or~\cite[Prop.~3.2.5]{NP21}). Then $M$ is \emph{rigid analytically trivial} if there exists $\Upsilon \in \GL_r (\TT)$ that satisfies
\begin{equation} \label{E:RATeq}
 \Upsilon^{(1)} = \Theta\Upsilon.
\end{equation}
The matrix $\Upsilon$ is called a \emph{rigid analytic trivialization} for $M$ with respect to~$\Theta$.

Returning to the situation  of our Drinfeld module $\phi$ from~\eqref{E:DMdef} and its associated $t$-motive $M(\phi) = \KK[\tau]$, one checks that $\{ 1, \tau, \dots, \tau^{r-1}\}$ forms a $\KK[t]$-basis for $M(\phi)$ (e.g., see \cite[\S 4.1]{Anderson86}, \cite[\S 5.4]{Goss}, \cite[\S 7.3]{Thakur}). If we let $\bm = (1, \tau, \cdots \tau^{r-1})^{\mathrm{tr}}$, then it follows (e.g., see \cite[Ex.~4.6.7]{NP21}, \cite[\S 4.2]{Pellarin08}) that $\tau \bm = \Theta\bm$, where $\Theta$ is defined in~\eqref{E:Theta}.

\begin{proposition}[{Pellarin~\cite[\S 4.2]{Pellarin08}}] \label{P:RAT}
Let $\phi : \bA \to \KK[\tau]$ be a Drinfeld module of rank~$r$.  Let $\pi_1, \dots, \pi_r$ be a basis of $\Lambda_\phi$, and for $1 \leqslant j \leqslant r$, let $f_j \assign f_{\phi}(\pi_j;t)$.  Then
\[
\Upsilon \assign \begin{pmatrix}
f_1    &   f_2         &  \cdots      &   f_r \\
f_1^{(1)} &   f_2^{(1)}   &  \cdots      &   f_r^{(1)} \\
\vdots    &   \vdots      &  \ddots      &   \vdots  \\
f_1^{(r-1)} &   f_2^{(r-1)}   &  \cdots      &   f_r^{(r-1)}
\end{pmatrix} \in \GL_r(\TT)
\]
is a rigid analytic trivialization for $\phi$ with respect to $\Theta$ in~\eqref{E:Theta}.
\end{proposition}

The proof of the functional equation follows from applications of~\eqref{E:AGFfnq}.  To show that $\det \Upsilon \in \TT^{\times}$ involves showing that it is a constant multiple of $\omega_C$ (see for example \cite[\S 6]{GezmisP19} or \cite[Prop.~4.3.10]{NP21} (note that `$\Upsilon$' in \cite{NP21} would be `$\Upsilon^{(1)}$' in the present paper)).  The reader is directed to \cite[\S 4]{Pellarin08} (or \cite[\S 3.4]{CP12}, \cite[\S 4.3]{NP21}) for more details.

When $r=2$, Pellarin~\cite[\S 2]{Pellarin12} showed that, as we vary $\phi$, we can realize $\Upsilon$ in terms of vector valued Drinfeld modular forms, which led to special value identities for what are now called Pellarin $L$-series.

\section{\texorpdfstring{$t$}{t}-power torsion and matrix estimates} \label{S:MatEst}
We fix a Drinfeld module $\phi : \bA \to \KK[t]$ as in~\eqref{E:DMdef}, together with the data assembled in \S\ref{S:Prelim}, including $\Theta \in \Mat_r(\KK[t])$ from~\eqref{E:Theta}.  The main goal of this section (Theorem~\ref{T:B}) is to construct a matrix $B \in \Mat_r(\KK[t]) \cap \GL_r(\TT)$ so that
\begin{equation} \label{E:one}
      \dnorm{B^{-1} \Theta^{-1} B^{(1)} - I} <1.
\end{equation}
We further strive for this construction to rely on only a finite amount of initial calculation of $t$-power torsion points of~$\phi$. The purpose of~\eqref{E:one} is the following. We first note that since $\det \Theta = \pm(t-\theta)/A_r \in \TT^{\times}$, we have $\Theta \in \GL_r(\TT)$.  By letting $F \assign B^{-1}\Theta^{-1}B^{(1)} \in \GL_r(\TT)$, we obtain that
\[
\dnorm{F^{(n)}-I} = \dnorm{F-I}^{q^n} \to 0, \quad n \to \infty.
\]
Thus the infinite product $\prod_{n=1}^\infty F^{(n)}$ converges in $\Mat_r(\TT)$ with respect to the Gauss norm, and its determinant is in~$\TT^{\times}$. And so by defining $\Pi \assign B F F^{(1)} F^{(2)} \cdots \in \GL_r(\TT)$,
\begin{equation} \label{E:PiRAT}
  \Pi^{(1)} = B^{(1)} F^{(1)} F^{(2)} \cdots = \Theta B F F^{(1)} F^{(2)}\cdots = \Theta\Pi,
\end{equation}
i.e., $\Pi$ is a rigid analytic trivialization for $\phi$ with respect to~$\Theta$ (see Corollary~\ref{C:PiRAT}).

To construct $B$ we make judicious choices of polynomials in $\KK[t]$ of the form
\[
  h = \phi_{t^{\ell-1}}(\xi) + \phi_{t^{\ell-2}}(\xi) t + \cdots + \xi t^{\ell-1}, \quad \xi \in \phi[ t^{\ell}].
\]
There are two things to note about such polynomials. The first is that if $\phi_{t^{\ell-1}}(\xi) \in \phi[t]$ is nonzero, then there is some nonzero period $\pi \in \Lambda_\phi$ so that $\phi_{t^{\ell-1}}(\xi) = \exp_\phi(\pi/\theta)$.  Thus as elements of $\TT$,
\begin{equation} \label{E:hcong}
  h \equiv f_{\phi}(\pi;t) \pmod{t^{\ell}},
\end{equation}
and so $h$ is a truncation of an Anderson generating function. Though, as $\pi$ is not uniquely determined by $\xi$, neither is $f_{\phi}(\pi;t)$. Furthermore, taking $\phi_t$ to be an operator on~$\TT$ as in \eqref{E:Deltaop}, one quickly checks that
\begin{equation} \label{E:phi(h)}
  \phi_t(h) - th = -\xi t^{\ell},
\end{equation}
which is a truncated version of~\eqref{E:AGFfnq}. Ultimately we will consider such polynomials for $\xi_1, \dots, \xi_r \in \phi[t^{\ell}]$ for which $\phi_{t^{\ell-1}}(\xi_1), \dots, \phi_{t^{\ell-1}}(\xi_r)$ form a \emph{strict basis} of $\phi[t]$ (see Definition~\ref{D:Strict}).  These $\FF_q$-bases of $\phi[t]$ arise in a directly similar manner to computations of Maurischat~\cite[Thm.~3.1]{Maurischat19a} for rank~$2$ Drinfeld modules and successive minimum bases of~$\Lambda_\phi$ defined by Gekeler~\cite[\S 1.3]{Gekeler19}.

\begin{proposition} \label{P:F-I}
Let $\ell \geqslant 1$. For $j=1,\dots,r$, we fix $\xi_j \in \phi[t^{\ell}]$, and let $h_j \assign \sum_{m=0}^{\ell-1} \phi_{t^{\ell-1-m}}(\xi_j)t^m$. Let
\begin{equation} \label{E:B}
B \assign \begin{pmatrix}
           h_1     	& h_2      	&  \ldots  & h_r  \\
           h_1^{(1)}   	& h_2^{(1)}  	&  \ldots  & h_r^{(1)}  \\
           \vdots 	&  \vdots 	&  \ddots  &  \vdots \\
          h_1^{(r-1)} & h_2^{(r-1)} &  \ldots  &  h_r^{(r-1)}
          \end{pmatrix},
\end{equation}
and assume $\det B \neq 0$. Then
\[
B^{-1}\Theta^{-1}B^{(1)} - I =  -\frac{t^{\ell}}{t-\theta} B^{-1} W,
\]
where
\begin{equation}  \label{E:W}
W =  \begin{pmatrix}
\xi_1   &  \xi_2   &  \cdots    &   \xi_r   \\
0       &     0    &  \cdots    &   0       \\
\vdots  & \vdots   &  \ddots    &  \vdots   \\
0       &     0    &  \cdots    &   0
\end{pmatrix}.
\end{equation}
\end{proposition}

\begin{proof}
We first observe that
\begin{align} \label{E:matrixidentity}
 (t-\theta)\Theta^{-1} B^{(1)} &= \begin{pmatrix}
		    A_1    &     A_2    &  \cdots    &   A_{r-1}    &   A_r \\
                 t-\theta   &     0      &  \cdots    &   0          &   0 \\
                     0      &  t-\theta  &  \cdots    &   0          &   0 \\
                   \vdots   &   \vdots  & \ddots     &  \vdots      &  \vdots  \\
                     0      &     0      &   \cdots   & t-\theta     &   0
\end{pmatrix}
\begin{pmatrix}
           h_1^{(1)} & h_2^{(1)}  &  \ldots  &h_r^{(1)}  \\
           h_1^{(2)} & h_2^{(2)}  &  \ldots  &h_r^{(2)}  \\
           \vdots &   \vdots &  \ddots  &  \vdots \\
           h_1^{(r)} & h_2^{(r)}  &  \ldots  &h_r^{(r)}
\end{pmatrix}\\
     &=
\begin{pmatrix}
w_1                	 & w_2  			& \cdots 	& w_r \\
(t-\theta) h_1^{(1)} 	 & (t-\theta) h_2^{(1)} 	& \cdots 	& (t-\theta) h_r^{(1)} \\
\vdots                   &  \vdots  			&  \ddots  	&  \vdots \\
(t-\theta) h_1^{(r-1)}   & (t-\theta) h_2^{(r-1)}  	& \cdots 	& (t-\theta) h_r^{(r-1)} \\
\end{pmatrix} \notag
\end{align}
where $w_j= A_1 h_j^{(1)} + \cdots + A_r h_j^{(r)} $. For each $j=1,\dots,r$, \eqref{E:phi(h)} implies $w_j = \phi_{t}(h_j) - \theta h_j = (t-\theta)h_j - t^{\ell}\xi_j$. Thus
\[
 (t-\theta)\Theta^{-1} B^{(1)} = -t^{\ell} W +(t-\theta)B,
\]
and the result follows by multiplying through by $(t-\theta)^{-1}B^{-1}$.
\end{proof}

\begin{remark} \label{R:condition}
By the above proposition, if $B \in \GL_r(\TT)$ so in particular $\det(B)$ is invertible in $\TT$, we have $\dnorm{B^{-1}\Theta^{-1}B^{(1)} - I} =  \dnorm{t^{\ell}/(t-\theta) \cdot B^{-1}W}$. Since $\dnorm{t^{\ell}/(t-\theta)} = 1/q$, it follows that proving $\dnorm{B^{-1}\Theta^{-1}B^{(1)} - I}  < 1$ is equivalent to showing
\begin{equation} \label{E:norm<q}
    \dnorm{B^{-1}W} < q.
\end{equation}
The remainder of this section is devoted to finding $\ell \geqslant 1$ and $\xi_1, \dots, \xi_r \in \phi[t^{\ell}]$ so that $B \in \GL_r(\TT)$ and this inequality holds.
\end{remark}

Suppose that $B$ has been chosen as in Proposition~\ref{P:F-I} and that $B \in \Mat_r(\KK[t]) \cap \GL_r(\TT)$.  For each $j$, we note that if
\[
|\xi_j| < |\phi_t(\xi_j)| < \dots < |\phi_{t^{\ell-1}}(\xi_j)|,
\]
then we obtain $\dnorm{h_j} = |\phi_{t^{\ell-1}}(\xi_j)|$. This will facilitate estimating $\dnorm{B^{-1}W}$ and prompts our next investigations.

Consider the Newton polygon $\Gamma$ (see Figure~\ref{F:NewtPoly}) of the polynomial
\[
\phi_t(x) = \theta x + A_1 x^q +\cdots + A_r x^{q^r} \in \KK[x].
\]
Letting $s \geqslant 1$ be the number of edges of $\Gamma$, we denote its vertices by $(q^{d_j},{-\deg A_{d_j}})$ for $0 \leqslant j \leqslant s$. We note that $0 = d_0 < d_1 < \dots < d_s = r$ and that each $d_j$ is an element of $N(\phi)$ as defined in \S\ref{S:Prelim}. For $0 \leqslant n < m \leqslant r$, define $L_{n,m}$ to be the line segment connecting vertices $(q^n,-\deg A_n)$ and $(q^m,-\deg A_m)$ and let $w_{n,m}$ be its slope. For $1 \leqslant j \leqslant s$, let $\lambda_j \assign w_{d_{j-1},d_j}$. We observe that $\lambda_1 < \lambda_2 < \dots < \lambda_s$ and that the line segments $L_{d_{0},d_1}, \dots, L_{d_{s-1},d_s}$ form the edges of $\Gamma$. For more information on Newton polygons see \cite[\S 2]{Goss} or \cite[\S VI.1.6]{Robert}.

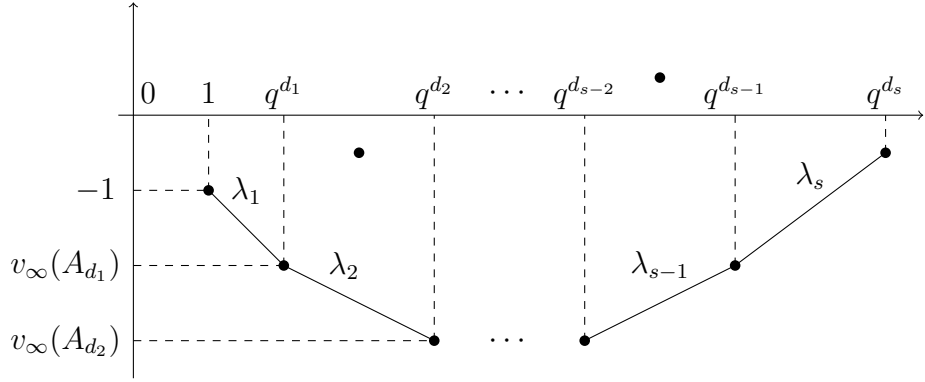
\begin{figure}[tb]
\begin{tikzpicture}
    \draw[->] (-0.2,0) -- (10.5,0);
    \draw[->] (0,-3.5) -- (0,1.5);

    \draw[domain=1:2] plot (\x,{- \x}) ;
    \draw[domain=2:4] plot (\x,{-1- 0.5 * \x}) ;
    \draw[domain=6:8] plot (\x,{-2+ 0.5 * (\x-8)});
    \draw[domain=8:10] plot (\x,{-2+  0.75*(\x-8)}) ;

    \node at (1.5,-1) {$\lambda_1$};
    \node at (2.8,-2) {$\lambda_2$};
    \node at (7,-2) {$\lambda_{s-1}$};
    \node at (9,-0.8) {$\lambda_s$};

    \node at (0.2,0.3) {$0$};
    \node at (1,0.3) {$1$};
    \node at (2,0.3) {$q^{d_1}$};
    \node at (4,0.3) {$q^{d_2}$};
    \node at (5,0.3) {$\cdots$};
    \node at (6,0.3) {$q^{d_{s-2}}$};
    \node at (8,0.3) {$q^{d_{s-1}}$};
    \node at (10,0.3) {$q^{d_s}$};

    \node at (-0.5,-1) {$-1$};
    \node at (-0.9,-2) {$v_\infty(A_{d_1})$};
    \node at (-0.9,-3) {$v_\infty(A_{d_2})$};

    \node at (5,-3) {$\cdots$};

    \draw[dashed] (0,-1) -- (1,-1);
    \draw[dashed] (0,-2) -- (2,-2);
    \draw[dashed] (0,-3) -- (4,-3);

    \draw[dashed] (1,-1) -- (1,0);
    \draw[dashed] (2,-2) -- (2,0);
    \draw[dashed] (4,-3) -- (4,0);
    \draw[dashed] (6,-3) -- (6,0);
    \draw[dashed] (8,-2) -- (8,0);
    \draw[dashed] (10,-0.5) -- (10,0);

   \foreach \coord in {
  {(1,-1)},
  {(2,-2)},
  {(3,-0.5)},
  {(4,-3)},
  {(6,-3)},
  {(7,0.5)},
  {(8,-2)},
  {(10,-0.5)}
   } {
  \fill \coord circle (2pt) ;
   }
\end{tikzpicture}
\caption{Newton polygon $\Gamma$ of $\phi_t(x)$.}
\label{F:NewtPoly}
\end{figure}

\begin{lemma} \label{L:poly}
Recall the definitions of $\mu_n$ for $n\in N(\phi)$ and $\mu_m$ from~\eqref{E:mu_n} and~\eqref{E:radius}. For $1 \leqslant j \leqslant s$, let $a_j$ be the $y$-intercept of the line containing $L_{d_{j-1},d_j}$. The following hold.
\begin{enumerate}
\item $a_1 = \mu_m$.
\item $a_1 > a_2 > \dots > a_s$.
\item $-a_j \geqslant -\left( \dfrac{\deg A_{d_j} - q^{d_j} }{q^{d_j} -1} \right)$ for every $j=1, \dots ,s$.
\end{enumerate}
\end{lemma}

\begin{proof}
Recalling the definitions of $\mu_n$ for $n\in N(\phi)$, we observe that $\mu_n$ is the $y$-intercept of the line through the points $(-1,1)$ and $(q^n,-\deg A_n)$. Due to the convexity of the Newton polygon, we have $\mu_{d_1}\geqslant \mu_m$. On the other hand, it follows from the definition of $\mu_m$ that $\mu_{d_1}\leqslant \mu_m$. So $\mu_{d_1} = \mu_m$, and it is clear that $a_1=\mu_{d_1}$ which proves (1).
For part~(2), since $\lambda_n < \lambda_{n+1}$, we have that $a_n > a_{n+1}$ (see Figure~\ref{F:SucMin}). To prove~(3), we recognize that again due to the convexity of the Newton polygon, the $y$-intercept $a_j$ must be at most the $y$-intercept $\mu_{d_j}$. Therefore,
\[
    -a_j \geqslant -\left( \frac{\deg A_{d_j} - q^{d_j} }{q^{d_j} -1} \right)
\]
for every $j=1, \dots ,s$.
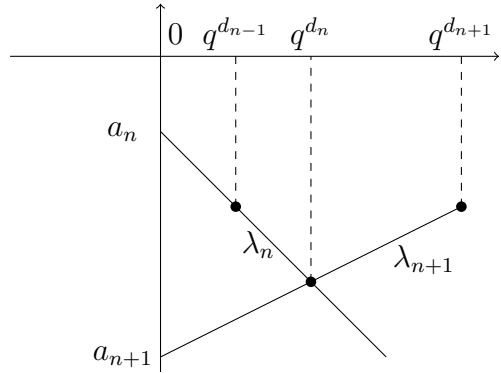
\begin{figure}[tb]
   \begin{tikzpicture}
    \draw[->] (-2,0) -- (4.5,0);
    \draw[->] (0,-4.2) -- (0,0.7);

    \draw[domain=0:3] plot (\x,{- \x-1}) ;
    \draw[domain=0:4] plot (\x,{-4+ 0.5 * \x}) ;

    \node at (1.3,-2.5) {$\lambda_n$};
    \node at (3.5,-2.7) {$\lambda_{n+1}$};
    \node at (0.2,0.3) {0};
    \node at (1,0.3) {$q^{d_{n-1}}$};
    \node at (2,0.3) {$q^{d_n}$};
    \node at (4,0.3) {$q^{d_{n+1}}$};
    \node at (-0.5,-1) {$a_n$};
    \node at (-0.5,-4) {$a_{n+1}$};

    \draw[dashed] (1,-2) -- (1,0);
    \draw[dashed] (2,-3) -- (2,0);
    \draw[dashed] (4,-2) -- (4,0);

   \foreach \coord in {
  {(1,-2)},
  {(2,-3)},
  {(4,-2)}
   } {
  \fill \coord circle (2pt) ;
   }
\end{tikzpicture}
\caption{Demonstrating $a_n > a_{n+1}$.}
\label{F:SucMin}
\end{figure}
\end{proof}

Building on the preceding lemma, we produce an algorithm to choose a sequence of torsion points $y_1$, $y_2, \ldots$ on $\phi$ that form a $t$-division sequence above~$0$ (that is convergent in the sense of \cite[\S 2.5.3]{HartlJuschka20}).

\begin{proposition} \label{P:algo}
Given nonzero $y_1\in \phi[t]$, there exists a recursive algorithm to choose $y_2$, $y_3, \ldots \in \KK$ such that
\begin{enumerate}
\item $\phi_t(y_k) = y_{k-1}$ for $k \geqslant 2$,
\item $\deg y_1 > \deg y_2  > \deg y_3 > \cdots$,
\item there exists a positive integer $N$ such that  $|y_N| < R_\phi$,
\item $\lim_{k\to\infty} \deg y_k  = -\infty$,
\item for $k\geqslant N$, $y_k$ is uniquely determined by $y_N$.
\end{enumerate}
\end{proposition}

\begin{proof}
Since $y_1$ is a root of $\phi_t(x)$, we see that $\deg y_1 \leqslant \lambda_s$, as $\lambda_s$ is the slope of the final segment of the Newton polygon~$\Gamma$. For $k\geqslant 1$, we perform the following recursive process. Suppose $\deg y_k \leqslant \lambda_s$ and set $y \assign y_k$. Consider the Newton polygon of $\phi_t(x) - y$, which is obtained from $\Gamma$ by adding one more point $(0,-\deg y)$. We observe that $-\deg y$ must belong to one of the following intervals:
\[
I_1 \assign (a_1,\infty), \quad I_2 \assign (a_2,a_1], \quad \ldots, \quad I_s \assign (a_s,a_{s-1}],
\]
where $a_1,\dots,a_s$ are defined in Lemma~\ref{L:poly}. To see why $-\deg y > a_s$, we claim that $-\lambda_s > a_s$, which follows from the following observation.
For a line through the point $(1,-1)$ the sum of its $y$-intercept and its slope is $-1$. If the line runs below the point $(1,-1)$, this sum is less than $-1$. We conclude that $a_s+\lambda_s \leqslant -1$, which makes $a_s < -\lambda_s $. Since $-\deg y \geqslant -\lambda_s$, we have $-\deg y > a_s$ as claimed.

Now for $n \in N(\phi)$, we define the rational function $u_n(z) \assign (z-\deg A_n)/q^n$.
\begin{enumerate}
\item[(i)]  If $-\deg y \in (a_1,\infty)$, then the Newton polygon of $\phi_t(x) - y$ is obtained from $\Gamma$ by adding the line segment from $(0,-\deg y)$ to $(1,-1)$. This new segment has slope $\deg y -1 = u_0(\deg y)$, so there is exactly one root of $\phi_t(x) - y$ with degree equal to $u_0(\deg y)$.
\item[(ii)]  If $-\deg y \in (a_{j+1},a_j]$ for some $1\leqslant j \leqslant s-1$, then the Newton polygon of $\phi_t(x) - y$ is obtained from $\Gamma$ by replacing line segments $L_{d_0,d_1}$, $L_{d_1,d_2}, \dots L_{d_{j-1},d_j}$ by the line segment from $(0,-\deg y)$ to $(q^{d_j},-\deg A_{d_j})$. This new segment has slope $(\deg y - \deg A_{d_j})/q^{d_j} = u_{d_j}(\deg y)$, and so there are $q^{d_j}$ roots of $\phi_t(x) - y$ with degree equal to $u_{d_j}(\deg y)$.
\end{enumerate}
Choose $y_{k+1}$ to be a root of $\phi_t(x) - y$ with
\[
 \deg y_{k+1} =
\begin{cases}
  u_0(\deg y) & \textup{if $-\deg y \in (a_1,\infty)$,} \\
  u_{d_j}(\deg y) &  \textup{if $-\deg y \in (a_{j+1},a_j]$.}
\end{cases}
\]

First we prove the inequality
\begin{equation} \label{E:ykineq}
\deg y_{k+1} \leqslant  \deg y_{k}-1.
\end{equation}
For the first case we note simply that $u_0(\deg y_{k}) = \deg y_{k} - 1$. For the second case, note that $\deg y_{k+1} = u_{d_j}(\deg y_{k})$ only if $-\deg y_{k} \in (a_{j+1},a_j]$ for some $1 \leqslant j \leqslant s-1$. Lemma~\ref{L:poly}(3) then implies that $\deg y_k \geqslant -a_j \geqslant  (q^{d_j} - \deg A_{d_j})/(q^{d_j} -1)$, and after some straighforward manipulations we deduce $\deg y_{k+1} = (\deg y_{k} - \deg A_{d_j})/q^{d_j} \leqslant \deg y_{k} -1$. In summary, in all cases we obtain a root $y_{k+1}$ of $\phi_t(x) - y_k$, which satisfies $\deg y_{k+1} \leqslant \deg y_k -1 <  \deg y_k \leqslant \lambda_s$. This proves~\eqref{E:ykineq}, as well as sets up the next step in the recursion, proving~(1) and~(2). The inequality in~\eqref{E:ykineq} also readily proves~(3) and~(4). By~\eqref{E:radius} and Lemma~\ref{L:poly}(1), we see that once $\deg y_k < -a_1$, we are in case (i) above, and moreover $y_{k+1}$ is uniquely determined by $y_k$. Thus, taking $N$ sufficiently large as in~(3), we obtain~(5).
\end{proof}

\begin{remark}  \label{R:N}
To make the recursion in Proposition~\ref{P:algo} effective, we further determine~$N$ explicitly from the given data. That is, we determine $N$ so that $\deg y_N < -a_1$. From~\eqref{E:radius} and~\eqref{E:ykineq}, we see that this will occur when
\[
N > \deg y_1 - \log_q(R_\phi) + 1.
\]
However, with some small effort this can in general be improved. Continuing with the notation in the proof,
we first note that
\begin{equation} \label{E:uncirck}
u_n^{\circ k} (z) = \frac{z}{q^{nk}} - \left(\frac{q^{nk}-1}{q^n-1}\right) \frac{\deg A_n}{q^{nk}},
\end{equation}
where $u_n^{\circ k}(z)$ is the $k$-th iterate of $u_n$ under composition. If $-\deg y_1 \in I_1 = (a_1,\infty)$, then we already have $\deg y_1 < -a_1$ and we can choose $N=1$.
Now assume that $-\deg y_1 \in I_{j+1}$ for some $1\leqslant j\leqslant s-1$. By Lemma \ref{L:poly},
\[
\lim_{k\to\infty} u_{d_j}^{\circ k} \bigl( \deg y_1 \bigr) = -\frac{\deg A_{d_j}}{q^{d_j}-1}  < -a_j ,
\]
so there exists a smallest integer $k_1$ so that $u_{d_j}^{\circ k_1} (\deg y_1) < -a_j$. Therefore $-\deg y_{k_1 +1} = -u_{d_j}^{\circ k_1} (\deg y_1) \in I_1 \cup \cdots \cup I_{j}$. Repeating the same argument, we can choose the smallest integer $k_2\geqslant0$ that makes $u_{d_{j-1}}^{\circ k_2} (\deg y_{k_1 +1}) < -a_{j-1}$, and thus
\[
-\deg y_{k_2 + k_1 +1} = -u_{d_{j-1}}^{\circ k_2} (\deg y_{k_1 +1}) \in I_1 \cup \cdots \cup I_{j-1}.
\]
Continuing inductively, we finally obtain $k_j\geqslant 0$ with $u_{d_1}^{\circ k_j} (\deg y_{k_{j-1} + \cdots + k_1 +1}) < -a_1$, i.e.,
\[
-\deg y_{k_j + \cdots + k_1 +1} = -u_{d_j}^{\circ k_j} (\deg y_{k_{j-1} + \cdots + k_1 +1}) \in I_1.
\]
Letting
\begin{equation} \label{E:Ndef}
N = 1+ k_1 + \cdots + k_j,
\end{equation}
we thus obtain $-\deg y_N \in I_1$, i.e., $\deg y_N < -a_1$.
Moreover, we observe that
\[
\deg y_N = u_{d_1}^{\circ k_j} \circ  \cdots \circ u_{d_{j-1}}^{\circ k_2} \circ u_{d_j}^{\circ k_1} (\deg y_1).
\]
\end{remark}

In the next proposition we explain how the recursive process in Proposition~\ref{P:algo} can be used to produce $\xi_1,\dots,\xi_r$ that are suitable inputs for Proposition~\ref{P:F-I}. Although a sequence $\{y_k\}$ obtained from Proposition~\ref{P:algo} is infinite, we construct $\xi_1, \dots, \xi_r$ after only finitely many applications of the recursive algorithm defining such sequences.

\begin{proposition} \label{P:good-xi}
Let $x_1,\dots,x_r$ be an $\FF_q$-basis of $\phi[t]$. Then there exists $N\geqslant 1$ and $\xi_1, \dots, \xi_r \in \phi[t^N]$ such that for each $j=1, \dots, r$,
\begin{enumerate}
\item $|\xi_j| < R_\phi$,
\item $\phi_{t^{N-1}}(\xi_j) = x_j$,
\item $\deg \phi_{t^{N-1}}(\xi_j) >  \dots > \deg \phi_t(\xi_j) > \deg \xi_j$.
\end{enumerate}
\end{proposition}

\begin{proof}
For $1 \leqslant j \leqslant r$, apply Proposition~\ref{P:algo} to $x_j$, obtaining a sequence $x_{j,1}$, $x_{j,2}, \ldots$ with the designated properties ($x_{j,1} = x_j$).  We let $N_j$ be the positive integer from Proposition~\ref{P:algo}(3), e.g., from~\eqref{E:Ndef}. Let $N \assign \max\{N_1, \dots, N_r\}$, and for $1 \leqslant j \leqslant r$, let $\xi_j \assign x_{j,N}$. Then for each~$j$, Proposition~\ref{P:algo} implies
\[
\phi_{t^{k}}(\xi_j) = x_{j,N-k}, \quad 0 \leqslant k \leqslant N-1.
\]
This proves~(2).  Parts~(1) and~(3) are then consequences of Proposition~\ref{P:algo}(2)--(3).
\end{proof}

\begin{definition} \label{D:Strict}
An $\FF_q$-basis $x_1, \dots, x_r$ of $\phi[t]$ is \emph{strict} if for $1\leqslant j \leqslant r$, we have $\deg x_j = \lambda_k$, where $d_{k-1}+1\leqslant j\leqslant d_k$.
\end{definition}

\begin{lemma}[{cf.\ Gekeler~\cite[Prop.~1.4]{Gekeler19}}] \label{L:Strict}
There exists a strict $\FF_q$-basis of $\phi[t]$.
\end{lemma}

\begin{proof}
For $1\leqslant k \leqslant s$, define
\begin{equation} \label{E:Qk}
Q_k \assign \{ x\in \phi[t]: \deg x \leqslant \lambda_k \}, \quad
R_k \assign \{ x\in \phi[t]: \deg x =  \lambda_k \}.
\end{equation}
We observe that $Q_k$ is an $\FF_q$-subspace of $\phi[t]$ for each $k$. Since $Q_1 = R_1\sqcup \{0\}$ and the set $Q_1$ has $q^{d_1}$ elements, there exist $x_1,\dots,x_{d_1} \in R_1$ such that
\[
Q_1 = \FF_q x_1 \oplus \dots \oplus \FF_q x_{d_1}.
\]
Now $Q_1 \subseteq Q_2$ and the set $Q_2$ has $q^{d_2}$ elements, so we can pick $x_{d_1 + 1},\dots,x_{d_2}$ in $Q_2$ such that
\[
Q_2 = \FF_q x_1 \oplus \dots \oplus \FF_q x_{d_1} \oplus \FF_q x_{d_1 + 1} \oplus \dots \oplus \FF_q x_{d_2}.
\]
We claim that for every $d_1 + 1 \leqslant j \leqslant d_2$, we have $\deg x_j = \lambda_2$. Indeed fix $d_1 + 1 \leqslant j \leqslant d_2$ and suppose that $\deg x_j < \lambda_2$. Since $x_j$ is a nonzero element in $\phi[t]$, we have $\deg x_j \in \{\lambda_1,\ldots,\lambda_s\}$, and $\deg x_j < \lambda_2$ implies that $\deg x_j \leqslant \lambda_1$. But then $x_j\in Q_1$, violating the $\FF_q$-linear independence of $x_1, \dots, x_{d_2}$. We continue inductively to produce an $\FF_q$-basis $x_1, \dots, x_{d_s}$ of $Q_s$, with $\{ x_{d_{k-1}+1}, \dots, x_{d_{k}} \} \subseteq R_k$ for each $1 \leqslant k \leqslant s$.  As $d_s = r$ and $Q_s = \phi[t]$, we are done.
\end{proof}

\begin{definition} \label{D:data}
For the remainder of this section, we fix the following data.  Pick a strict basis $x_1, \dots, x_r$ of $\phi[t]$. Choose $N\geqslant 1$ and $\xi_1, \dots, \xi_r \in \phi[t^N]$ as in Proposition~\ref{P:good-xi}.  For $1\leqslant j \leqslant r$, we set
\[
   h_j \assign \phi_{t^{N-1}} (\xi_j) + \phi_{t^{N-2}} (\xi_j) t  + \cdots +  \xi_j t^{N-1} \in \KK[t].
\]
Proposition~\ref{P:good-xi}(3) implies that
\begin{equation} \label{E:h_jnorm}
\dnorm{h_j} = \max_{0\leqslant m \leqslant N-1} \bigl\{ |\phi_{t^{N-1-m}} (\xi_j)| \bigr\} = |x_j|
\end{equation}
and moreover $|x_j|$ is the unique maximum among the norms of the coefficients of $h_j$. Fix $B$ to be the $r\times r$ matrix with $(i,j)$-entries defined by $[B]_{ij} \assign h_j^{(i-1)}$ as in~\eqref{E:B}. For the rest of this section, our goal is to prove that~$B$ is in $\GL_r(\TT)$ and satisfies \eqref{E:norm<q}.  Thus by Proposition~\ref{P:F-I} (and Remark~\ref{R:condition}), $B$ will satisfy~\eqref{E:one} as desired.
\end{definition}

\begin{lemma} \label{L:P_s}
Let $x_1,\dots,x_r$ be a strict basis of $\phi[t]$.  For $1\leqslant k \leqslant s$ and $d_{k-1}+1\leqslant j\leqslant d_k$,
\[
 \deg(c_1 x_1 + \dots + c_{j-1} x_{j-1}+  x_j ) = \lambda_k = \deg x_j,
\]
for every $c_1,\dots, c_{j-1} \in \FF_q$.
\end{lemma}

\begin{proof}
Let $y = c_1 x_1 + \dots + c_{j-1} x_{j-1} +  x_j$, and let $z = c_{d_{k-1}+1}x_{d_{k-1}+1} + \cdots + c_{j-1} x_{j-1} + x_j$. We observe that $\deg y \leqslant \deg x_j = \lambda_k$. If  $\deg y \leqslant  \lambda_{k-1}$, then it must be the case that $\deg z \leqslant \lambda_{k-1}$, whence $z \in Q_{k-1}$ as in~\eqref{E:Qk}.  As $x_1, \dots, x_{d_{k-1}}$ is an $\FF_q$-basis of $Q_{k-1}$, this would violate the $\FF_q$-linear independence of $x_1, \dots, x_j$.
\end{proof}

\begin{lemma} \label{L:detX}
Let $x_1 , \dots, x_r$ be a strict basis for $\phi[t]$, and fix $X \in \Mat_r(\KK)$ so that $[X]_{ij} = x_j^{q^{i-1}}$. Then
\[
\deg(\det X) = \sum_{j=1}^r q^{j-1}\deg x_j.
\]
\end{lemma}

\begin{proof}
As $X$ is a Moore matrix, we obtain from~\cite[Cor.~1.3.7]{Goss} that
\[
\det X =  \prod_{j=1}^r \prod_{c_1,\ldots, c_{j-1}\in \FF_q} (c_1 x_1 + \cdots + c_{j-1} x_{j-1} + x_j ).
\]
For $1\leqslant k \leqslant s$, let $P_k = \{d_{k-1}+1, \dots , d_k \}$. Thus by Lemma~\ref{L:P_s},
\[
\deg(\det X)
= \sum_{k=1}^s \sum_{j\in P_k} \sum_{c_1,\ldots, c_{j-1}\in \FF_q} \deg x_j
= \sum_{k=1}^s \sum_{j\in P_k} q^{j-1}\deg x_j,
\]
which easily converts to the desired formula.
\end{proof}

The following modification of the identity in Lemma~\ref{L:detX} can be obtained by the same argument and will be useful in Proposition~\ref{P:detB}. For $1 \leqslant j_1 < \cdots < j_{n} \leqslant r$ and $1 \leqslant i \leqslant n$, consider the $n \times n$ matrix $\tX$ so that $[\tX]_{i\ell} = \bigl( x_{j_\ell}^{q^{i-1}} \bigr)$. Then
\begin{equation} \label{E:dettX}
  \deg(\det \tX) = \sum_{\ell=1}^n q^{\ell-1} \deg x_{j_{\ell}}.
\end{equation}
We can now compute $\det B$ and show that $B \in \GL_r(\TT)$.

\begin{proposition} \label{P:detB}
For $1\leqslant j \leqslant r$, write $h_j = x_j + y_j t$ with $y_j \in \KK[t]$. Let $B = (h_j^{(i-1)})$ and let $X = \bigl( x_j^{q^{i-1}} \bigr)$. Then
\begin{enumerate}
\item $\det B = \det X + y t$ for some $y \in \KK[t]$ with $\dnorm{y} < \norm{\det X}$,
\item $B \in \GL_r(\TT)$.
\end{enumerate}
\end{proposition}

\begin{proof}
Using the calculation in~\eqref{E:h_jnorm} via Proposition~\ref{P:good-xi}(3), we see that $\dnorm{y_j} < |x_j|$ and so
\[
  h_j^{(i-1)} \in \TT^{\times}, \quad 1 \leqslant i,\, j \leqslant r.
\]
For $1 \leqslant n \leqslant r$, let $\tB$ be any $n \times n$ minor of $B$, where any of the columns of $B$ are removed and all of the final $r-n$ rows of $B$ are removed.  Thus we can choose $1 \leqslant j_1 < \cdots < j_n \leqslant r$ so that $[\tB]_{i\ell} = (h_{j_{\ell}}^{(i-1)})$. If $n=r$, then $\tB$ is simply $B$. Recalling $\tX$ from~\eqref{E:dettX}, we claim that for some $\ty \in \KK[t]$,
\[
  \det \tB = \det \tX + \ty t,
\]
with $\dnorm{\ty} < \norm{\det \tX}$, and moreover that $\det \tB \in \TT^{\times}$.

We proceed by induction on $n$.  When $n=1$, the claim follows since $h_j \in \TT^{\times}$ for each~$j$.  Consider the case $n=2$. Then suppose $h_{j_1} = \tx_1 + \ty_1 t$ and $h_{j_2} = \tx_2 + \ty_2 t$. Then
\[
 \det \tB
 = (\tx_1 + \ty_1 t)( \tx_2^{(1)} + \ty_2^{(1)} t) - (\tx_2 + \ty_2 t)( \tx_1^{(1)} + \ty_1^{(1)} t)
 = \det \tX + \ty t,
\]
where $\ty = \ty_1 h_{j_2}^{(1)} + \tx_1  \ty_2^{(1)} - \ty_2 h_{j_1}^{(1)} - \tx_2 \ty_1^{(1)} \in \KK[t]$.  We quickly verify that
\begin{gather*}
  \dnorm{\ty_1 h_{j_2}^{(1)}} < |\tx_1| \cdot |\tx_2|^q, \quad
  \dnorm{\tx_1 \ty_2^{(1)}} < |\tx_1| \cdot |\tx_2|^q, \\
  \dnorm{\ty_1 h_{j_1}^{(1)}} < |\tx_2| \cdot |\tx_1|^q, \quad
  \dnorm{\tx_2 \ty_1^{(1)}} < |\tx_2| \cdot |\tx_1|^q.
\end{gather*}
Combining these inequalities with~\eqref{E:dettX}, we have $\dnorm{\ty} < \norm{\det \tX}$, and thus $\det \tB \in \TT^{\times}$.

The general argument is similar, though there is more bookkeeping.  Suppose the statement is true for $n-1<r$, and for $1\leqslant \ell \leqslant n$ write $h_{j_\ell} = \tx_{\ell} + \ty_{\ell} t$.
We can express
\begin{equation} \label{E:dettB}
\det \tB = \sum_{\ell=1}^n (-1)^{n+\ell} h_{j_\ell}^{(n-1)} \det B_\ell,
\end{equation}
where $B_\ell$ is the $(n-1)\times (n-1)$-minor matrix obtained by removing the last row and the $\ell$-th column from $\tB$. We observe that
\[
B_\ell =
\begin{pmatrix}
   h_{j_1}      & \cdots   &  h_{j_{\ell-1}}       & h_{j_{\ell+1}}    & \cdots & h_{j_n} \\
   \vdots   & \ddots   & \vdots          & \vdots  & \ddots     & \vdots\\
   h_{j_1}^{(n-2)} & \cdots & h_{j_{\ell-1}}^{(n-2)} & h_{j_{\ell+1}}^{(n-2)} & \cdots & h_{j_\ell}^{(n-2)}
\end{pmatrix}.
\]
By the induction hypothesis, $\det B_\ell = \det X_\ell + b_\ell t$ for some $b_\ell \in \KK[t]$ and $\dnorm{b_\ell} < \norm{\det X_\ell}$ (where $X_\ell$ is the evident $(n-1)\times (n-1)$-minor matrix of $\tX$).  Starting from~\eqref{E:dettB},
\begin{align*}
  \det \tB
  &= \sum_{\ell=1}^n (-1)^{n+\ell} \bigl(\tx_\ell^{q^{n-1}} + \ty_\ell^{(n-1)} t \bigr) (\det X_\ell + b_\ell t) \\
  &= \sum_{\ell=1}^n (-1)^{n+\ell} \bigl(\tx_\ell^{q^{n-1}} \det(X_\ell)  + c_\ell t \bigr), \\
\intertext{where $c_\ell \assign \ty_\ell^{(n-1)}\cdot \det B_\ell + \tx_\ell^{q^{n-1}}b_\ell$, and continuing,}
  &= \det \tX  + \sum_{\ell=1}^n (-1)^{n+\ell} c_\ell t.
\end{align*}
We claim that for each $\ell$, $\dnorm{c_\ell} < \norm{\det \tX}$. Considering the definition of $c_\ell$, we estimate
\[
  \lVert \ty_\ell^{(n-1)} \cdot \det B_\ell \rVert < |\tx_\ell|^{q^{n-1}} \cdot \norm{\det X_\ell}, \quad
  \lVert \tx_\ell^{q^{n-1}} \cdot b_\ell \rVert < |\tx_\ell|^{q^{n-1}} \cdot \norm{\det X_\ell}.
\]
The first inequality follows from $\dnorm{\ty_\ell} < |\tx_\ell|$ and $\dnorm{\det B_\ell} = \norm{\det X_\ell}$, and the second from $\dnorm{b_\ell} < \norm{\det X_\ell}$.  Finally, $|\tx_\ell|^{q^{n-1}} \cdot \norm{\det X_\ell} \leqslant \norm{\det \tX}$ by~\eqref{E:dettX}.  Thus this claim and the induction are complete.
\end{proof}

\begin{remark} \label{R:detBdetX}
From the above proposition, we see that $\dnorm{\det B} = \norm{\det X}$, and for each minor $\tB$ in the proof, $\dnorm{\det \tB} = \norm{\det \tX}$.
\end{remark}

We now prove the main theorem of this section, which verifies that the matrix $B$ we have constructed in Proposition~\ref{P:detB} satisfies~\eqref{E:one}.

\begin{theorem}  \label{T:B}
Let $x_1,\dots,x_r$ be a strict basis of $\phi[t]$. Choose $N\geqslant 1$ and $\xi_1,\dots,\xi_r \in \phi[t^N]$ as in Proposition~\ref{P:good-xi}, and for $1\leqslant j\leqslant r$, define
\[
h_j \assign \phi_{t^{N-1}} (\xi_j) + \phi_{t^{N-2}} (\xi_j) t  + \dots +  \xi_j t^{N-1}.
\]
Let $B \assign (h_j^{(i-1)}) \in \Mat_r(\KK[t])$ as in~\eqref{E:B}. Then $B \in \GL_r(\TT)$ and
\[
\dnorm{B^{-1}\Theta^{-1}B^{(1)} - I} < 1.
\]
\end{theorem}

\begin{proof}
That $B \in \GL_r(\TT)$ was proved in Proposition~\ref{P:detB}. By Remark~\ref{R:condition}, we thus see that $\dnorm{B^{-1}\Theta^{-1}B^{(1)} - I}  < 1$ if and only if $\dnorm{B^{-1}W} < q$, where $W$ is defined in~\eqref{E:W} as
\begin{equation*} 
W =  \begin{pmatrix}
\xi_1   &  \xi_2   &  \cdots    &   \xi_r   \\
0       &     0    &  \cdots    &   0       \\
\vdots  & \vdots   &  \ddots    &  \vdots   \\
0       &     0    &  \cdots    &   0
\end{pmatrix}.
\end{equation*} Letting $m_{ij} \assign [B^{-1}]_{ij}$, we find
\[
B^{-1}W
= \begin{pmatrix}
m_{11}\xi_1    &  m_{11}\xi_2    &  \cdots    &   m_{11}\xi_r   \\
m_{21}\xi_1    &  m_{21}\xi_2    &  \cdots    &   m_{21}\xi_r   \\
\vdots         &    \vdots       &  \ddots    &  \vdots         \\
m_{r1}\xi_1    &  m_{r1}\xi_2    &  \cdots    &   m_{r1}\xi_r
\end{pmatrix}.
\]
By definition $\dnorm{B^{-1}W} = \max\{ \dnorm{ m_{i1}\xi_j } : 1 \leqslant i, j \leqslant r \}$. Now fix $i$ and $j$. Expressing $B^{-1}$ in terms of cofactors, we see that $m_{i1} = (-1)^{i+1} (\det B^*)/(\det B)$, where
\[
B^* = \begin{pmatrix}
 h_1^{(1)}   &   \cdots &  h_{i-1}^{(1)}   &  h_{i+1}^{(1)}  &   \cdots  &  h_r^{(1)}  \\
 h_1^{(2)}   &   \cdots &  h_{i-1}^{(2)}   &  h_{i+1}^{(2)}  &   \cdots  &  h_r^{(2)}  \\
 \vdots      &   \ddots &  \vdots	   &  \vdots 	     &  \ddots   &   \vdots  \\
 h_1^{(r-1)}   &   \cdots &  h_{i-1}^{(r-1)}   &  h_{i+1}^{(r-1)}  &   \cdots  &  h_r^{(r-1)}
\end{pmatrix}.
\]
We note that $B^* = \tB^{(1)}$, where $\tB$ is the $(r-1)\times (r-1)$ minor of $B$ obtained by removing the $i$-th column and last row, whose properties were investigated in the proof of Proposition~\ref{P:detB}.  In particular, if as in that proof we let $\tX$ be the corresponding minor of $X$, we see from Remark~\ref{R:detBdetX} that
\[
  \dnorm{\det B} = \norm{\det X}, \quad \dnorm{\det B^*} = \dnorm{\tB^{(1)}} = \norm{\det \tX}^q.
\]
We can use Lemma~\ref{L:detX} to calculate $\norm{\det X}$ and~\eqref{E:dettX} to calculate $\norm{\det \tX}$, and we find
\begin{align} \label{E:mi1xijnorm1}
 \log_q \dnorm{m_{i1} \xi_j} &= \log_q \biggl( \frac{\dnorm{\det B^*}\cdot |\xi_j| }{\dnorm{\det B }} \biggr) \\
&= \sum_{\ell=1}^{i-1} q^\ell \deg x_{\ell} + \sum_{\ell=i+1}^r q^{\ell-1} \deg x_\ell + \log_q |\xi_j| - \sum_{\ell=1}^r q^{\ell-1} \deg x_\ell. \notag
\end{align}
Using that $\deg x_\ell \leqslant \deg x_{\ell+1}$ for all $\ell \leqslant r-1$, after reindexing the sum we find
\[
\sum_{\ell=1}^{i-1} q^\ell \deg x_\ell \leqslant \sum_{\ell=2}^i q^{\ell-1} \deg x_\ell.
\]
Combining this with~\eqref{E:mi1xijnorm1}, we have
\begin{equation} \label{E:mi1xijnorm2}
  \log_q \dnorm{m_{i1} \xi_j} \leqslant -\deg x_1 + \log_q |\xi_j|.
\end{equation}
Since $\xi_1,\dots,\xi_r$ are chosen so that $|\xi_j| < R_\phi = q^{-\mu_m} $, we have $\log_q |\xi_j| < -\mu_m$.
Recall from Lemma~\ref{L:poly}(1) and its proof that
\[
\mu_m = a_1 = p_{d_1} = -1-w_{0,d_1} = -1-\lambda_1 = -1 - \deg x_1.
\]
Continuing with~\eqref{E:mi1xijnorm2}, we then find
\[
  \log_q \dnorm{m_{i1} \xi_j} < -\deg x_1 + 1 + \deg x_1 = 1.
\]
Thus $\log_q \dnorm{B^{-1}W} = \max_{i,j} ( \log_q \dnorm{m_{i1} \xi_j}) < 1$, and so $\dnorm{B^{-1}W} < q$ as sought.
\end{proof}

Based on the discussion at the beginning of this section and in particular~\eqref{E:PiRAT}, the following corollary is immediate.

\begin{corollary} \label{C:PiRAT}
Continuing with the setting of Theorem~\ref{T:B}, the infinite product
\[
  \Pi \assign B \prod_{n=0}^\infty \Bigl( B^{-1} \Theta^{-1} B^{(1)} \Bigr)^{(n)}
\]
converges with respect to the Gauss norm on $\Mat_r(\TT)$ and lies in~$\GL_r(\TT)$.  Moreover,
\[
  \Pi^{(1)} = \Theta \Pi,
\]
and so $\Pi$ is a rigid analytic trivialization for~$\phi$.
\end{corollary}

\section{Product expansions for rigid analytic trivializations} \label{S:prod}
We continue with the constructions and notations of \S\ref{S:MatEst}, especially the results from Theorem~\ref{T:B} and Corollary~\ref{C:PiRAT}.  A natural question to ask is how the matrix $\Pi$ from Corollary~\ref{C:PiRAT} compares to $\Upsilon$ defined by Pellarin in Proposition~\ref{P:RAT}.  By applying results from~\cite{EP14}, we will show in Theorem~\ref{T:Pi} that once we define an appropriate basis $\pi_1, \dots, \pi_r$ for $\Lambda_\phi$, using our already chosen $\xi_1, \dots, \xi_r \in \phi[t^N]$, that in fact $\Pi = \Upsilon$ (cf.\ \cite[\S 3]{KPhD}).

\begin{lemma} \label{L:Pi_n}
Set $F \assign B^{-1} \Theta^{-1} B^{(1)}$. For $n \geqslant 0$, set $\Pi_n \assign B F F^{(1)} \cdots F^{(n)}$.  Let $W$ be the matrix defined in~\eqref{E:W}, and let
\[
R_0 \assign \frac{I}{t-\theta}, \quad R_m \assign \frac{\Theta^{-1} (\Theta^{-1})^{(1)} \cdots (\Theta^{-1})^{(m-1)} }{t-\theta^{q^m}}, \quad m\geqslant 1.
\]
Then for $n\geqslant 0$,
\[
\Pi_n = B - t^N \sum_{m=0}^n R_m W^{(m)} .
\]
\end{lemma}

\begin{proof}
First for $n\geqslant 0$, we observe that
\[
\Pi_n = \Theta^{-1} \bigl( \Theta^{-1} \bigr)^{(1)} \cdots \bigl( \Theta^{-1} \bigr)^{(n)} B^{(n+1)}.
\]
Moreover, Proposition~\ref{P:F-I} yields
\[
\Theta^{-1}B^{(1)} = B  - \frac{t^N}{t-\theta} W,
\]
which proves the case $n=0$. Proceeding by induction, assume the formula holds for $\Pi_{n-1}$.  Then combining the previous displayed equations,
\[
\Pi_n = \Theta^{-1} \bigl( \Theta^{-1} \bigr)^{(1)} \cdots \bigl( \Theta^{-1} \bigr)^{(n-1)} \biggl( B^{(n)} - \frac{t^N}{t-\theta^{q^n}} W^{(n)} \biggr)
= \Pi_{n-1} - t^N R_n W^{(n)},
\]
and the formula for $\Pi_n$ follows from the induction hypothesis.
\end{proof}

In the previous lemma, we wrote $\Pi_n$ as a finite sum involving the matrices $B$, $W$ and $R_m$, $0\leqslant m \leqslant n$. Next we investigate an explicit formula for the first column of $R_m$. For $n \geqslant 0$, recall the rational function $\cB_n \in \KK(t)$ from~\eqref{E:Bndef}.  We set $\cB_n \assign 0$ for $n<0$.

\begin{lemma} \label{L:R_m}
For $m \geqslant 0$, let $R_m$ be defined as in Lemma~\ref{L:Pi_n}. Then for $1 \leqslant i \leqslant r$,
\[
[R_m]_{i1} = \frac{\cB_{m-(i-1)}^{(i-1)}}{t-\theta^{q^{i-1}}}.
\]
\end{lemma}

\begin{proof}
We proceed by induction on $m$. As $\cB_0 = 1$, the conclusion holds trivially for $R_0$. For general $m$ we observe that $R_{m} = \Theta^{-1}R_{m-1}^{(1)}$, and so by combining the exact form of $\Theta^{-1}$ (see~\eqref{E:matrixidentity}), the induction hypothesis, and~\cite[Lem.~6.12(a)]{EP14},
\[
  [R_m]_{11} = \sum_{k=1}^r \bigl[ \Theta^{-1} \bigr]_{1k} \bigl[ R_{m-1}^{(1)} \bigr]_{k1}
  = \sum_{k=1}^r\frac{A_k}{t-\theta} \frac{\cB_{m-k}^{(k)}}{t-\theta^{q^k}}
  =  \frac{\cB_{m}}{t-\theta} .
\]
The cases $2 \leqslant i \leqslant r$ follow immediately from the induction hypothesis, as
\[
  [R_m]_{i1} = \sum_{k=1}^r \bigl[ \Theta^{-1} \bigr]_{ik} \bigl[ R_{m-1}^{(1)} \bigr]_{k1}
  = [R_{m-1}^{(1)}]_{i-1,1}
  = \frac{\cB_{m-(i-1)}^{(i-1)}}{t-\theta^{q^{i-1}}}.
  \qedhere
\]
\end{proof}

Combining Lemmata~\ref{L:Pi_n} and~\ref{L:R_m}, we determine the matrix $\Pi_n$ completely.

\begin{proposition}  \label{P:Pi_n}
For $n\geqslant 0$ and $1\leqslant i$, $j \leqslant r$,
\[
[\Pi_n]_{ij} = \biggl( h_j  - \frac{t^N}{t-\theta}\sum_{m=0}^{n-(i-1)} \cB_m \xi_j^{q^m} \biggr)^{(i-1)}.
\]
\end{proposition}

\begin{proof}
Using~\eqref{E:W} and Lemma~\ref{L:R_m}, we obtain that for each $m$,
\[
\bigl[ R_m W^{(m)} \bigr]_{ij}
= [R_m]_{i1} \xi_j^{q^m} =  \frac{\cB_{m-(i-1)}^{(i-1)}}{t-\theta^{q^{i-1}}}\cdot \xi_j^{q^m}.
\]
Substituting this into the formula for $\Pi_n$ in Lemma~\ref{L:Pi_n}, it follows that
\[
 [\Pi_n]_{ij}
 = h_j^{(i-1)} - t^N \sum_{m=0}^n \frac{\cB_{m-(i-1)}^{(i-1)}}{t-\theta^{q^{i-1}}} \cdot \xi_j^{q^m}
 = \biggl( h_j -  t^N \sum_{m=i-1}^{n} \frac{\cB_{m-(i-1)}}{t-\theta} \cdot \xi_j^{q^{m-(i-1)}} \biggr)^{(i-1)},
\]
and the result follows by reindexing the sum.
\end{proof}

The following theorem determines the entries of $\Pi$ in terms of $\cL_{\phi}(\xi_j;t)$ from~\eqref{E:Lxit}, starting from a strict basis $x_1, \dots, x_r$ of $\phi[t]$. As a consequence we find that $\Pi$ is the same as the matrix $\Upsilon$ constructed from Anderson generating functions in Proposition~\ref{P:RAT}.

\begin{theorem} \label{T:Pi}
Let $x_1, \dots, x_r$ be a strict basis of $\phi[t]$, and choose $N \geqslant 1$ and $\xi_1, \dots, \xi_r \in \phi[t^N]$ as in Proposition~\ref{P:good-xi}. Let $B = ( h_j^{(i-1)}) \in \GL_r(\TT)$ be defined as in Theorem~\ref{T:B}, and construct the rigid analytic trivialization for $\phi$,
\[
\Pi = B \prod_{n=0}^\infty \Bigl( B^{-1} \Theta^{-1} B^{(1)} \Bigr)^{(n)}
\]
as in Corollary~\ref{C:PiRAT}.  For $1 \leqslant i,\,j \leqslant r$,
\[
 [\Pi]_{ij} = \biggl( h_j - \frac{t^N}{t-\theta} \cL_{\phi}(\xi_j;t) \biggr)^{(i-1)}.
\]
Moreover, letting $\pi_j \assign \theta^N \log_{\phi}(\xi_j)$, the quantities $\pi_1, \dots, \pi_r$ form an $A$-basis of $\Lambda_{\phi}$ and $[\Pi]_{ij} = f_{\phi}(\pi_j;t)^{(i-1)}$. It follows that
\[
  \Pi = \Upsilon,
\]
where $\Upsilon$ is defined with respect to $\pi_1, \dots, \pi_r \in \Lambda_\phi$ in Proposition~\ref{P:RAT}.
\end{theorem}

\begin{proof}
Recall in Proposition~\ref{P:good-xi} that $\xi_1, \dots, \xi_r$ are chosen so that $|\xi_j|< R_{\phi}$ for each~$j$. Thus as in~\eqref{E:Lxit}, we have $\cL_{\phi}(\xi_j;t) \in \TT$ for each $j$ by~\cite[Prop.~6.10]{EP14}.  For $1\leqslant i,\,j \leqslant r$ fixed, it then follows from~\eqref{E:Lxit} and Proposition~\ref{P:Pi_n} that
\[
  [\Pi]_{ij} = \lim_{n\to \infty} \biggl( h_j - \frac{t^N}{t-\theta} \sum_{m=0}^{n-(i-1)} \cB_m \xi_j^{q^m} \biggr)^{(i-1)} = \biggl( h_j - \frac{t^N}{t-\theta} \cL_{\phi} (\xi_j; t) \biggr)^{(i-1)},
\]
thus verifying the first identity.  As $R_{\phi}$ is at most the radius of convergence of $\log_{\phi}(z)$ (see~\cite[Thm.~6.13(b)]{EP14}, but also Corollary~\ref{C:Rphi} below), $\log_{\phi}(\xi_j)$ is well-defined for each~$j$, and so we can define $\pi_j \assign \theta^N \log_{\phi}(\xi_j)$ ($\Rightarrow \xi_j = \exp_\phi(\pi_j/\theta^N)$). By Proposition~\ref{P:good-xi},
\[
 \exp_\phi(\pi_j) = \exp_\phi \bigl( \theta^N \log_\phi(\xi_j) \bigr) = \phi_{t^N}(\xi_j)=0,
\]
so $\pi_1,\dots,\pi_r \in \Lambda_\phi$. As in~\eqref{E:hcong}, since $\phi_{t^{N-1}}(\xi_j)= x_j = \exp_\phi(\pi_j/\theta)$, we have $h_j = \sum_{m=0}^{N-1} \exp_\phi (\pi_j/\theta^{m+1}) t^m$.
By \eqref{E:LxitAGF}, we find $\cL_{\phi}(\xi_j;t) = -(t-\theta) f_{\phi}(\pi_j/\theta^N;t)$, and using the definition of Anderson generating function in~\eqref{E:AGFdef}, we obtain (cf.\ \cite[Eq.~(7.2)]{EP14})
\[
  h_j - \frac{t^N}{t-\theta} \cL_{\phi}(\xi_j;t) = f_{\phi}(\pi_j;t).
\]

Thus for each $i$, $j$, we have $[\Pi]_{ij} = f_{\phi}(\pi_j;t)^{(i-1)}$, and so $\Pi$ has the form of $\Upsilon$ in Proposition~\ref{P:RAT}, but in order for that proposition to apply it remains to verify that $\pi_1, \dots, \pi_r$ is an $A$-basis of $\Lambda_{\phi}$. Let $\omega_1, \dots, \omega_r \in \Lambda_\phi$ be any $A$-basis of $\Lambda_\phi$, and let $U = ( f_{\phi}(\omega_j;t)^{(i-1)})$ be the corresponding matrix from~\eqref{E:Upsilon}. Suppose $E \in \Mat_r(A)$ satisfies
\[
 (\omega_1, \dots, \omega_r) E = (\pi_1, \dots, \pi_r),
\]
and let $\tE \in \Mat_r(\bA)$ be the corresponding matrix where $\theta \mapsto t$. From~\eqref{E:AGFscalar}, it follows that $U\tE=\Pi$, whence $\det U\cdot \det \tE = \det \Pi$.  Since $\Pi$, $U \in \GL_r(\TT)$ (by Proposition~\ref{P:RAT} and Corollary~\ref{C:PiRAT}), it must be that $\det \tE \in \TT^{\times}$. However, $\TT^{\times} \cap \bA = \FF_q^{\times}$, and so $\det E = \det \tE \in \FF_q^{\times}$. Therefore, $\pi_1, \dots, \pi_r$ form an $A$-basis of~$\Lambda_\phi$.
\end{proof}

\begin{corollary} \label{C:Rphi}
The radius of convergence of $\log_{\phi}(z)$ is exactly~$R_{\phi}$.
\end{corollary}

\begin{proof}
As in \eqref{E:lograd}, the radius of convergence $P_{\phi}$ of $\log_{\phi}(z)$ is the minimum norm among nonzero periods in $\Lambda_\phi$ by \cite[Prop.~4.14.2]{Goss}.  Furthermore, $R_\phi \leqslant P_\phi$, as explained in Remark~\ref{R:Rphi}. From our strict basis for $\phi[t]$, it follows that $\deg x_1 = \lambda_1$, and Lemma~\ref{L:poly} implies $\lambda_1 = -1 - \mu_m$. Therefore, $|x_1| = R_{\phi}/|\theta| < R_{\phi}$, and it follows from the proof of Theorem~\ref{T:Pi} that $\pi_1 = \theta\log(x_1)$. Also by \cite[Prop.~4.14.2]{Goss}, $\log_{\phi}(z)$ is an isometry on the open disk of radius $P_{\phi}$ centered at~$0$, and so $|\pi_1| = |\theta|\cdot |x_1| = R_{\phi}$. Thus, $R_{\phi} \geqslant P_{\phi}$.
\end{proof}

These results also shed light on the field generated by the period lattice. Let $L/k_\infty$ be a finite extension, containing the coefficients $A_1, \dots, A_r$. When $r=2$, Maurischat~\cite[Thm.~3.1]{Maurischat19a}, proved that $L(\Lambda_{\phi}) = L(\phi[t^N])$, using Newton polygons in a similar fashion to Propositions~\ref{P:algo} and~\ref{P:good-xi} (see Theorem~\ref{T:Rank2}). For higher ranks, Gekeler~\cite[\S 2]{Gekeler19} used the spectrum of successive minimum bases for $\Lambda_{\phi}$ to extend Maurischat's result. Although not providing the degree of detail as Gekeler's result, we obtain the following for all ranks, whose proof has been designed in a similar manner to Maurischat's rank~$2$ result.

\begin{corollary} \label{C:fieldexts}
Suppose that $\phi$ is defined over a finite extension $L/k_{\infty}$. Choose $N \geqslant 1$ as in Remark~\ref{R:N} and Proposition~\ref{P:good-xi}.  Then $L(\Lambda_{\phi}) = L(\phi[t^N])$.
\end{corollary}

\begin{proof}
As finite extensions of $k_\infty$, $L$ and $L(\phi[t^N])$ are complete. By~\cite[Prop.~2.1]{Maurischat19a}, we have that $L(\phi[t^N]) \subseteq L(\Lambda_{\phi})$. By definition of the $A$-basis $\pi_1, \dots, \pi_r$ of $\Lambda_{\phi}$ from Theorem~\ref{T:Pi}, $\pi_j = \theta^N \log_\phi(\xi_j)$. Since $\log_\phi(z)$ has coefficients in $L$, we have $\log_\phi(\xi_j) \in L(\phi[t^N])$ for each~$j$.
\end{proof}

\begin{remark}
It is an interesting question, brought up by one of the referees, about how much the preceding constructions and results depend on whether the initial $\FF_q$-basis $x_1, \dots, x_r \in \phi[t]$ is strict or not. At first the authors suspected that one could use the results for strict bases to then deduce the same results for non-strict bases, but it appears to be fairly subtle, and certain conclusions may turn out to be false if the initial basis is not strict.

For example, we can look ahead to Example~\ref{Ex:A}, where $q=3$, $\phi$ has rank~$2$, and $N=2$, and use the torsion elements $x_1$, $x_2 \in \phi[t]$ and $\xi_1$, $\xi_2 \in \phi[t^2]$ from~\eqref{E:Exx1x2} and \eqref{E:Exxi1xi2}. If we take
\begin{gather*}
x_1' \assign x_1 + x_2, \quad x_2' \assign x_2 \quad \in \phi[t], \\
\xi_1' \assign \xi_1 + \xi_2 + x_1, \quad \xi_2' \assign \xi_2 \quad \in \phi[t^2],
\end{gather*}
then $x_1'$, $x_2'$ do not form a strict basis for $\phi[t]$, but the sequences $x_1'$, $\xi_1', \ldots$ and $x_2'$, $\xi_2', \ldots$ can be extended to $t$-division sequences as in Proposition~\ref{P:algo} with terms of strictly decreasing degree. However, when we consider the corresponding polynomials, $h_1' = x_1' + \xi_1't$, $h_2' = x_2' + \xi_2't$, the matrix $B' = \left(\begin{smallmatrix} h_1' & h_2' \\ {h_1'}^{(1)} & {h_2'}^{(2)} \end{smallmatrix} \right)$ has determinant in $\laurent{\FF_3}{\theta^{-1/2}}[t]$,
\[
\det B' = -\theta^{1/2} - (\theta^{1/2} + \theta^{-1/2} + \theta^{-5} + \cdots) \cdot t + (\theta^{-11/2} + \theta^{-15} + \cdots)\cdot t^2.
\]
The details of this computation follow quickly from~\eqref{E:Exx1x2} and~\eqref{E:Exxi1xi2}. What we notice is that the coefficient of $t$ in $\det B'$ has the same $\infty$-adic norm as the constant term, implying the conclusion from Proposition~\ref{P:detB}(1) does not hold if we start with the non-strict basis $x_1'$, $x_2'$.
\end{remark}

\section{Rank \texorpdfstring{$2$}{2} Drinfeld modules and examples} \label{S:Rank2}
We now specialize to the case that $\phi$ has rank $2$, defined by
\begin{equation} \label{E:DMrank2}
  \phi_t = \theta + A_1 \tau + A_2 \tau^2, \quad A_2 \neq 0.
\end{equation}
We seek to express the data from the previous sections explicitly, and we work out concrete examples. Though framed slightly differently, Theorem~\ref{T:Rank2} below was proved by Maurischat~\cite[Thm.~3.1]{Maurischat19a} using similar methods to Propositions~\ref{P:algo} and~\ref{P:good-xi}. It is interesting to extract it from these propositions, which we briefly summarize afterwards. The Newton polygon for $\phi_t(x)$ falls into two cases, depicted in Figure~\ref{F:Rank2}, where (1) it has a single bottom edge of slope $\lambda_1$, or (2) it has two bottom edges of slopes $\lambda_1< \lambda_2$. We note that $\phi_t(x)$ belongs to case~(1) if and only if $\deg A_1 \leqslant (q+\deg A_2)/(q+1)$.

\begin{figure}[tb]
\begin{tikzpicture}
     \draw[->] (-0.2,0) -- (4.5,0);
     \draw[->] (0,-3.5) -- (0,1.1);

    \draw[domain=1:4] plot (\x,{(-1/3)- (2/3)*\x}) ;

    \node at (2.5,-1.5) {$\lambda_1$};

    \node at (0.2,0.3) {$0$};
    \node at (1,0.3) {$1$};
    \node at (2,0.3) {$q$};
    \node at (4,0.3) {$q^2$};

    \node at (-0.5,-1) {$-1$};
    \node at (-0.9,-0.5) {$v_\infty(A_1)$};
    \node at (-0.9,-3) {$v_\infty(A_2)$};

    \node at (2,-3.5) {Case 1};

    \draw[dashed] (0,-1) -- (1,-1);
    \draw[dashed] (0,-0.5) -- (2,-0.5);
    \draw[dashed] (0,-3) -- (4,-3);

    \draw[dashed] (1,0) -- (1,-1);
    \draw[dashed] (2,0) -- (2,-0.5);
    \draw[dashed] (4,0) -- (4,-3);

   \foreach \coord in {
  {(1,-1)},
  {(2,-0.5)},
  {(4,-3)}
   } {
  \fill \coord circle (2pt) ;
   }
\end{tikzpicture}
\begin{tikzpicture}
    \draw[->] (-0.2,0) -- (4.5,0);
    \draw[->] (0,-3.5) -- (0,1.1);

    \draw[domain=1:2] plot (\x,{- \x}) ;
    \draw[domain=2:4] plot (\x,{-1- 0.5 * \x}) ;

    \node at (1.5,-1) {$\lambda_1$};
    \node at (2.8,-2) {$\lambda_2$};

    \node at (0.2,0.3) {$0$};
    \node at (1,0.3) {$1$};
    \node at (2,0.3) {$q$};
    \node at (4,0.3) {$q^2$};

    \node at (-0.5,-1) {$-1$};
    \node at (-0.9,-2) {$v_\infty(A_1)$};
    \node at (-0.9,-3) {$v_\infty(A_2)$};

    \node at (2,-3.5) {Case 2};

    \draw[dashed] (0,-1) -- (1,-1);
    \draw[dashed] (0,-2) -- (2,-2);
    \draw[dashed] (0,-3) -- (4,-3);

    \draw[dashed] (1,0) -- (1,-1);
    \draw[dashed] (2,0) -- (2,-2);
    \draw[dashed] (4,0) -- (4,-3);

   \foreach \coord in {
  {(1,-1)},
  {(2,-2)},
  {(4,-3)}
   } {
  \fill \coord circle (2pt) ;
   }
\end{tikzpicture}
\caption{Rank $2$ Newton polygon cases for $\phi_t(x)$.}
\label{F:Rank2}
\end{figure}
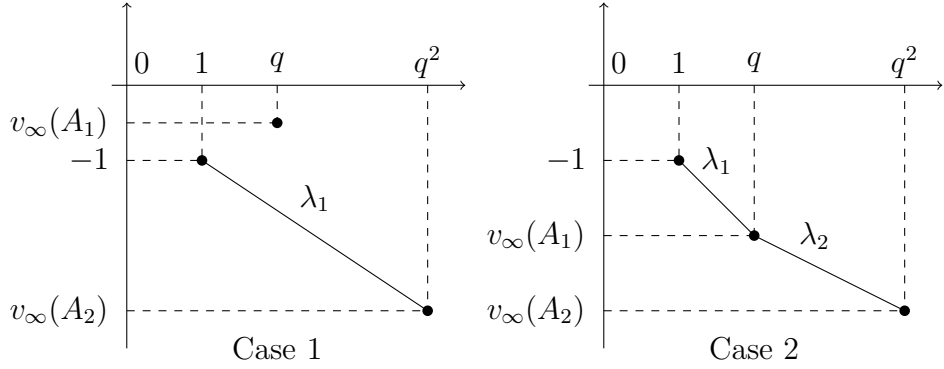

\begin{theorem}[{Maurischat~\cite[Thm.~3.1]{Maurischat19a}}] \label{T:Rank2}
Let $\phi$ be a rank $2$ Drinfeld module defined as in \eqref{E:DMrank2}, and consider the following cases.
\begin{enumerate}
  \item[(1)] $\deg A_1 \leqslant (q+\deg A_2)/(q+1)$,
  \item[(2)] $\deg A_1 > (q+\deg A_2)/(q+1)$.
\end{enumerate}
In Case~\textup{2}, let $\ell \geqslant 1$ be the unique integer such that
\[
\frac{q^{\ell} + \deg A_2}{q+1} \leqslant \deg A_1 < \frac{q^{\ell+1} + \deg A_2}{q+1}.
\]
Then the positive integer~$N$ and the degrees of $\xi_1$, $\xi_2 \in \phi[t^N]$ in Proposition~\ref{P:good-xi} can be chosen to satisfy the values in Figure~\ref{F:Cases}.
\begin{figure}[htb]
\renewcommand{\arraystretch}{1.25}
\begin{tabular}{ |c|c|c|c| }
 \hline
 \textup{Case} & $N$ & $\deg \xi_1$ & $\deg \xi_2$ \\
 \hline & & & \\[-10pt]
 \textup{1} & $1$ & $\dfrac{1-\deg A_2}{q^2-1}$ & $\dfrac{1-\deg A_2}{q^2-1}$ \\[10pt]
 \hline & & & \\[-10pt]
 \textup{2} & $\ell$ & $\dfrac{1-\deg A_1}{q-1} - (\ell-1)$ & $\dfrac{(-q^\ell+q+1)\deg A_1 - \deg A_2}{q^\ell(q-1)}$ \\[10pt]
 \hline
\end{tabular}
\caption{Degrees of $\xi_1$, $\xi_2$ in rank~$2$.}
\label{F:Cases}
\end{figure}
\end{theorem}

We briefly recount how Remark~\ref{R:N} and Proposition~\ref{P:good-xi} lead to these results in the case of rank~$2$. Let $x_1$, $x_2 \in \phi[t]$ be a strict basis.
(i) In Case~1, $\deg(x_1)=\deg(x_2)=\lambda_1$. As in the proof of Corollary~\ref{C:Rphi}, $|x_j| < R_\phi$ for $j=1$, $2$, and so $N=1$. (ii) In Case~2, $\deg(x_1) = \lambda_1$, $\deg(x_2) = \lambda_2$, and $\lambda_1 < \lambda_2$. Again as in Corollary~\ref{C:Rphi}, $|x_1|< R_{\phi}$, and so $N_1=1$. Now $-\lambda_2 \in (a_2,\infty]$, and we choose $\kappa \geqslant 0$ minimal so that $u_1^{\circ \kappa}(\lambda_2) < -a_1$. Using \eqref{E:uncirck} together with the definitions of~$a_1$ and~$\lambda_2$, a short calculation yields that $\kappa$ is minimal so that
\[
\deg A_1 < \frac{q^{\kappa+2} + \deg A_2}{q+1},
\]
from which $\kappa=\ell-1$ (cf.\ \cite[Pf.~of Thm.~3.1]{Maurischat19a}). Applying Remark~\ref{R:N}, we then have $N_2 = \kappa+1 = \ell$, and the degrees of $\xi_1$, $\xi_2$ follow.

As in Theorem~\ref{T:Pi}, the periods $\pi_1$, $\pi_2 \in \Lambda_\phi$ are defined by $\pi_j = \theta^N \log_{\phi}(\xi_j)$, for $j=1$,~$2$. As $\xi_1$, $\xi_2$ are strictly within the radius of convergence of $\log_\phi(z)$,
\begin{equation} \label{E:pideg}
    \deg \pi_j = N + \deg \xi_j, \quad j=1,\, 2,
\end{equation}
by \cite[Prop.~4.14.2]{Goss}. We now consider explicit examples to demonstrate how one can perform the constructions of $B$ and $\Pi$ in \S\ref{S:MatEst}--\S\ref{S:prod}. Although both examples below are Drinfeld modules with complex multiplication, the approximations that follow could be performed for any Drinfeld module, and it is not especially important that they have complex multiplication. On the other hand we found it enlightening to compare the calculations of various quantities with those previously obtained by other methods.

\begin{example} \label{Ex:Car}
We let $\phi : \bA \to \KK[\tau]$ be defined by
\[
 \phi_t = \theta + \bigl(\theta^{q/2} + \theta^{1/2} \bigr)\tau + \tau^2,
\]
which in essence is the Carlitz module for the ring $\FF_q[t^{1/2}] \supseteq \bA$. In this case $\deg A_1 = q/2$ and $\deg A_2 = 0$. From Theorem~\ref{T:Rank2}, we are in Case 2 and $N=\ell=1$. Let $\bsi$ be a choice of $(-1)^{1/(q-1)}$, and let $\FF = \FF_q(\bsi)$. Then in the field $\laurent{\FF}{\theta^{-1/(2q-2)}}$, we set
\[
  \zeta \assign \bsi \theta^{1/(2q-2)}, \quad
  \beta \assign \theta^{-1/2} + \theta^{-q/2} + \theta^{-q^2/2} + \theta^{-q^3/2} + \cdots.
\]
We verify that $x_1 \assign \zeta\beta$, $x_2 \assign \zeta$, form a strict basis of $\phi[t]$ with $\deg(x_1) = -(q-2)/(2q-2)$ and $\deg(x_2) = 1/(2q-2)$. The matrix $B$ from Theorem~\ref{T:B} is
\[
  B = \begin{pmatrix} \zeta \beta & \zeta \\ \zeta^q \beta^q & \zeta^q \end{pmatrix}
  = \zeta \begin{pmatrix} \beta & 1 \\ 1-\theta^{1/2}\beta & -\theta^{1/2} \end{pmatrix}.
\]
A small calculation reveals that
\[
 B^{-1} \Theta^{-1} B^{(1)} =
 \begin{pmatrix}
    1 - \dfrac{\theta^{1/2}\beta\, t}{t-\theta} & -\dfrac{\theta^{1/2}\,t}{t-\theta} \\[10pt]
    \dfrac{\theta^{1/2}\beta^{q+1}\, t}{t-\theta} & 1 + \dfrac{\theta^{1/2}\beta^q\, t}{t-\theta}
 \end{pmatrix},
\]
and we easily verify that $\dnorm{B^{-1}\Theta^{-1}B^{(1)}-I} < 1$ as in Theorem~\ref{T:B}. We omit the details, but it is instructive to compare approximations of residues and evaluations of $\Pi$ at $t=\theta$ to the expected periods and quasi-periods coming from the theory of the Carlitz module.
\end{example}

\begin{example} \label{Ex:A}
Let $q=3$, and choose $\nu \in \KK$ so that $\nu^2 = \theta^3 - \theta -1$. Fix the Drinfeld module $\phi: \bA \to \KK[\tau]$ defined by
\[
 \phi_t = \theta + (\nu^3 + \nu) \tau +  \tau^2.
\]
This Drinfeld module has been widely studied as it has complex multiplication by the class number~$1$ ring $\FF_3[t,y]$, where $y^2=t^3-t-1$ and
\[
\phi_y = \nu + (\nu^4-\nu^2)\tau + (\nu^9 + \nu^3 + \nu)\tau^2 + \tau^3.
\]
Thakur~\cite[\S 2.3(c)]{Thakur93}, \cite[\S 8.2]{Thakur}, showed that the coefficients of $\phi_t$, $\phi_y$, $\exp_{\phi}(z)$, and $\log_{\phi}(z)$, can be derived from its associated shtuka function, and this point of view was used in~\cite[\S 4, \S 9]{GreenP18} to find explicit product expansions of the direct analogues of $\omega_C$ and~$\tpi$.

First we observe that $\deg A_1 = \tfrac{9}{2}$ and $\deg A_2 = 0$, so $\lambda_1=-\tfrac{7}{4}$ and $\lambda_2=\tfrac{3}{4}$. Thus in Theorem~\ref{T:Rank2} and~\eqref{E:pideg} we have $N=\ell=2$ and
\[
  \deg \xi_1 = -\frac{11}{4}, \quad \deg \xi_2 = -\frac{5}{4}, \quad
  \deg \pi_1 = -\frac{3}{4}, \quad \deg \pi_2 = \frac{3}{4}.
\]
These agree with~\cite[\S 7]{EP13}, where some of these quantities were also calculated. In what follows all quantities are contained in $L\assign \laurent{\FF_9}{\theta^{-1/4}}$. We fix $\bsi \assign \sqrt{-1} \in \FF_9$. The sign function on $k_{\infty}^{\times} \to \FF_3^{\times}$ extends to $L^{\times} \to \FF_9^{\times}$ in the natural way, and we choose~$\nu$ to have sign~$+1$. Using Newton's method we approximate the strict basis $x_1$, $x_2$ for $\phi[t]$:
\begin{align} \label{E:Exx1x2}
  x_1 &= \begin{aligned}[t]
    -\bsi \theta^{-7/4} - \bsi \theta^{-19/4} - \bsi\theta^{-27/4} + \bsi \theta^{-31/4} - \bsi \theta^{-35/4}
    &{}+ \bsi \theta^{-39/4} \\
    &{}+ \bsi \theta^{-43/4} + O(\theta^{-51/4}),
  \end{aligned}
  \\
  x_2 &= \begin{aligned}[t]
    \bsi \theta^{3/4} - \bsi \theta^{-1/4} + \bsi \theta^{-5/4} + \bsi \theta^{-9/4}
    &{}- \bsi \theta^{-13/4} - \bsi\theta^{-17/4} + \bsi\theta^{-21/4} \\
    &{}+ \bsi \theta^{-33/4} - \bsi \theta^{-37/4} + \bsi\theta^{-41/4} + O(\theta^{-45/4}).
  \end{aligned} \notag
\end{align}
If desired, we can solve for $x_1$, $x_2$ algebraically and find
\begin{align*}
  x_1^2 &= \theta^{1/2}(\theta+1) - \Bigl( -(\theta-1)\nu\theta^{1/2} - \theta(\theta+1)^2 \Bigr)^{1/2}, \\
  x_2^2 &= \theta^{1/2}(\theta+1) + \Bigl( -(\theta-1)\nu\theta^{1/2} - \theta(\theta+1)^2 \Bigr)^{1/2}.
\end{align*}
Again using Newton's method, we approximate $\xi_1$, $\xi_2 \in \phi[t^2]$ as in Proposition~\ref{P:good-xi}:
\begin{align} \label{E:Exxi1xi2}
  \xi_1 &= \begin{aligned}[t]
  -\bsi \theta^{-11/4} - \bsi \theta^{-19/4} - \bsi \theta^{-23/4} + \bsi \theta^{-31/4} + \bsi \theta^{-35/4}
  &{}+ \bsi \theta^{-39/4}\\
  &{}+ \bsi \theta^{-43/4} + O(\theta^{-51/4}),
  \end{aligned}\\
  \xi_2 &= \begin{aligned}[t]
    -\bsi \theta^{-5/4} + \bsi \theta^{-9/4} &{}+ \bsi \theta^{-13/4} - \bsi \theta^{-17/4} + \bsi \theta^{-21/4} + \bsi \theta^{-25/4}  \\
    &{}- \bsi \theta^{-29/4} - \bsi \theta^{-33/4} - \bsi \theta^{-37/4} - \bsi \theta^{-45/4} + O(\theta^{-53/4}).
  \end{aligned} \notag
\end{align}
We then have $h_1 = x_1 + \xi_1 t$, $h_2 = x_2 + \xi_2 t$, from which we form $B$ as in Theorem~\ref{T:B}. Writing $F = B^{-1} \Theta^{-1} B^{(1)}$, we find
\begin{align}
  [F]_{11} &= \begin{aligned}[t]
    1 + \bigl( \theta^{-2} &{}+ \theta^{-4} + \theta^{-7} - \theta^{-8} - \theta^{-9} + \theta^{-10} + O(\theta^{-14}) \bigr) t^2 \\
    &{}+ \bigl( -\theta^{-5} - \theta^{-7} - \theta^{-8} + \theta^{-10} + \theta^{-12} + O(\theta^{-15}) \bigr) t^3 + O(t^4),
  \end{aligned}
  \\
  [F]_{12} &= \begin{aligned}[t]
    \bigl( \theta^{-1/2} &{}- \theta^{-3/2} - \theta^{-5/2} + \theta^{-13/2} + \theta^{-15/2} + \theta^{-17/2} + \theta^{-19/2} \\
    &{}+ \theta^{-21/2} + O(\theta^{-27/2}) \bigr) t^2 + \bigl( -\theta^{-7/4} + \theta^{-9/2} + \theta^{-11/2}  \\
    &{}- \theta^{-13/2} + \theta^{-15/2} + \theta^{-17/2} + \theta^{-19/2} + O(\theta^{-27/2}) \bigr) t^3 + O(t^4),
  \end{aligned} \notag
  \\
  [F]_{21} &= \bigl( \theta^{-19/2} + \theta^{-23/2} + O(\theta^{-25/2}) \bigr) t^2 + O(t^4), \notag \\
  [F]_{22} &= 1 + \bigl( \theta^{-8} - \theta^{-9} - \theta^{-10} + \theta^{-11} - \theta^{-12} + O(\theta^{-13}) \bigr)t^2 + O(t^4). \notag
\end{align}
Thus $B^{-1} \Theta^{-1} B^{(1)}$ conforms to the expected inequality in Theorem~\ref{T:B}. We can approximate $\Pi$ using its product expansion in Theorem~\ref{T:Pi}, and then we estimate the periods $\pi_1$, $\pi_2$, and corresponding quasi-periods $\eta_1$, $\eta_2$, for $\phi$ as described in \S\ref{S:Intro}. Namely,
$\pi_j = -((t-\theta)[\Pi]_{1j})|_{t=\theta}$ and $\eta_j = ([\Pi]_{2j})|_{t=\theta}$ for $j=1$, $2$. From this we find
\begin{align}
  \pi_1 &= \begin{aligned}[t]
  -\bsi \theta^{-3/4} - \bsi\theta^{-11/4} - \bsi\theta^{-15/4} - \bsi\theta^{-19/4} + \bsi\theta^{-23/4}
  &{}+ \bsi\theta^{-39/4} \\
  &{}+ \bsi\theta^{-47/4} + O(\theta^{-55/4}),
  \end{aligned}
  \\
  \pi_2 &= \begin{aligned}[t]
    -\bsi \theta^{3/4} + \bsi \theta^{-5/4} + \bsi \theta^{-9/4} - \bsi \theta^{-29/4}
    &{}- \bsi \theta^{-33/4} + \bsi \theta^{-37/4} \\
    &{}- \bsi\theta^{-41/4} + \bsi \theta^{-45/4} + O(\theta^{-61/4}).
  \end{aligned} \notag
\end{align}
We first note that these formulas approximate the identity $\pi_2 = \nu \pi_1$ at least to the precision taken, as one would expect in this case of complex multiplication. Moreover, though we omit the details, this approximation for $\pi_1$ agrees numerically with (i) the one due to Green and the second author~\cite[Thm.~4.6, Rem.~4.7]{GreenP18}, coming from shtuka functions, and (ii) the one of Gekeler and Hayes~\cite[Eq.~(7.10.6)]{Goss}, from the theory of sign-normalized Drinfeld modules. We further compute the quasi-periods,
\begin{align}
  \eta_1 &= \bsi \theta^{-21/4} + \bsi \theta^{-29/4} + \bsi \theta^{-37/4} + \bsi \theta^{-45/4} - \bsi \theta^{-53/4} + \bsi \theta^{-57/4} + O(\theta^{-61/4}), \\
  \eta_2 &= \begin{aligned}[t]
    -\bsi \theta^{9/4} + \bsi \theta^{-3/4} + \bsi \theta^{-11/4} - \bsi\theta^{-15/4}
    &{}+ \bsi \theta^{-19/4} -\bsi \theta^{-23/4} \\
    &{}+ \bsi\theta^{-39/4} + \bsi\theta^{-47/4} + O(\theta^{-51/4}).
  \end{aligned} \notag
\end{align}
Using~\eqref{E:tpi}, one checks that $\pi_1\eta_2 - \pi_2\eta_1$ and $\bsi\tpi$ agree to the given precision, which aligns with the Legendre relation~\cite[\S 2]{Goss94}, \cite[Thm.~6.4.6]{Thakur}.
\end{example}

\end{document}